\documentclass[12pt,oneside,english,british]{amsart}
\usepackage{helvet}

\usepackage[T1]{fontenc}
\usepackage[latin9]{inputenc}
\usepackage[a4paper]{geometry}
\usepackage{babel}
\usepackage{amsthm}
\usepackage{amstext}
\usepackage{graphicx}
\usepackage{esint}
\usepackage[unicode=true,
 bookmarks=true,bookmarksnumbered=false,bookmarksopen=false,
 breaklinks=false,pdfborder={0 0 1},backref=false,colorlinks=false]
 {hyperref}
\hypersetup{pdftitle={QnConvergence},
 pdfauthor={Paul Verschueren}}
\usepackage{breakurl}

\makeatletter
\numberwithin{equation}{section}
\numberwithin{figure}{section}
\usepackage{enumitem}		
\theoremstyle{plain}
\newtheorem{thm}{\protect\theoremname}[section]
  \theoremstyle{plain}
  \newtheorem{lem}[thm]{\protect\lemmaname}
  \theoremstyle{definition}
  \newtheorem{defn}[thm]{\protect\definitionname}
  \theoremstyle{plain}
  \newtheorem{cor}[thm]{\protect\corollaryname}
  \theoremstyle{plain}
  \newtheorem{conjecture}[thm]{\protect\conjecturename}

\usepackage[left,modulo]{lineno}

\makeatother

  \addto\captionsbritish{\renewcommand{\conjecturename}{Conjecture}}
  \addto\captionsbritish{\renewcommand{\corollaryname}{Corollary}}
  \addto\captionsbritish{\renewcommand{\definitionname}{Definition}}
  \addto\captionsbritish{\renewcommand{\lemmaname}{Lemma}}
  \addto\captionsbritish{\renewcommand{\theoremname}{Theorem}}
  \addto\captionsenglish{\renewcommand{\conjecturename}{Conjecture}}
  \addto\captionsenglish{\renewcommand{\corollaryname}{Corollary}}
  \addto\captionsenglish{\renewcommand{\definitionname}{Definition}}
  \addto\captionsenglish{\renewcommand{\lemmaname}{Lemma}}
  \addto\captionsenglish{\renewcommand{\theoremname}{Theorem}}
  \providecommand{\conjecturename}{Conjecture}
  \providecommand{\corollaryname}{Corollary}
  \providecommand{\definitionname}{Definition}
  \providecommand{\lemmaname}{Lemma}
\providecommand{\theoremname}{Theorem}

\begin{document}

\title{On the growth of Sudler's sine product $\prod_{r=1}^{n}|2\sin\pi r\omega|$
\\
at the golden rotation number}

\author{Paul Verschueren$\dagger$}

\thanks{$\dagger$Corresponding author: paul@verschueren.org.uk}

\author{Ben Mestel{*}}

\thanks{{*}Ben.Mestel@open.ac.uk}

\address{Department of Mathematics and Statistics}

\address{The Open University}

\address{Milton Keynes}

\address{MK7 6AA, UK}
\begin{abstract}
We study the growth at the golden rotation number $\omega=(\sqrt{5}-1)/2$
of the function sequence $P_{n}(\omega)=\prod_{r=1}^{n}|2\sin\pi r\omega|$.
This sequence has been variously studied elsewhere as a skew product
of sines, Birkhoff sum, q-Pochhammer symbol (on the unit circle),
and restricted Euler function. In particular we study the Fibonacci
decimation of the sequence $P_{n}$, namely the sub-sequence $Q_{n}=\left|\prod_{r=1}^{F_{n}}2\sin\pi r\omega\right|$
for Fibonacci numbers $F_{n}$, and prove that this renormalisation
subsequence converges to a constant. From this we show rigorously
that the growth of $P_{n}(\omega)$ is bounded by power laws. This
provides the theoretical basis to explain recent experimental results
reported by Knill and Tangerman (Self-similarity and growth in Birkhoff
sums for the golden rotation. Nonlinearity, 24(11):3115\textendash 3127,
2011).

\end{abstract}
\maketitle

\emph{AMS 2010:} 37D45, 26A12, 37E40, 41A60, 60G18

\section{Introduction}

\begin{figure}
\noindent \centering{}\includegraphics[scale=0.5]{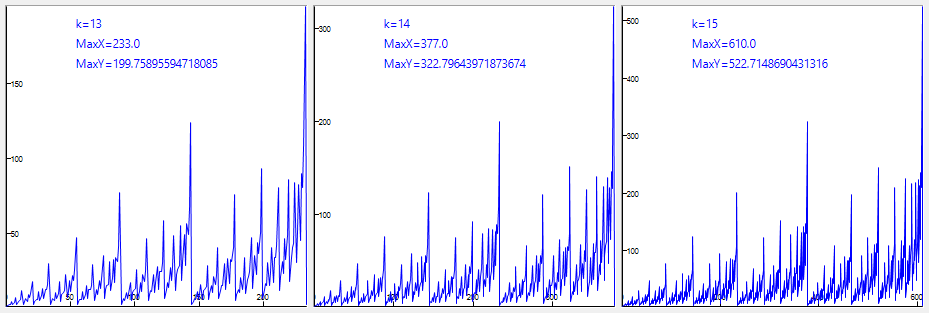}\protect\caption{\label{fig:Pnn}Renormalised graphs of $P_{n}(\omega)$ against $n$
in the range $1...F_{k}$. Notable features are the self-similarity
of the function and the almost linear growth of the peaks. Note also
that the peaks are substantially higher than other values of the function.}
\end{figure}

We will study the sequence of positive real functions $P_{n}(\omega)$
which we define as follows: 

\begin{equation}
P_{n}(\omega)=\prod_{r=1}^{n}|2\sin\pi r\omega|\label{eq:BasePn}
\end{equation}
This function sequence arises in a surprising number of fields of
pure and applied mathematics and physics, disguised by a number of
different representations and terminologies. In pure mathematics there
are important applications in partition theory (\cite{Sudler1964,Wright1964}),
Padé approximation and continued fractions (see \cite{Lubinsky1998}
for a list of 13 examples), whereas in applied mathematics the function
arises in the study of strange non-chaotic attractors (SNA's), KAM
theory and string theory (see \cite{Grebogi1984,Kuznetsov1995a,Knill2011,Awata01052013}).
For example in the context of SNA's (our own area of interest), the
functions \eqref{eq:BasePn} arise in the renormalisation analysis
of skew products. Figure \ref{fig:Pnn} shows renormalised graphs
of the function for various values of $n$ over the Fibonacci interval
$[1..F_{n}]$. Note the self-similarity, and also the approximately
linear linear growth of the peaks. The common problem in each of these
areas is to understand some aspect of the growth of the sequence.

This list of application areas is no doubt incomplete if only because
of the remarkable range of representations and terminology under which
the function sequence appears. The representation $P_{n}(\omega)=\prod_{r=1}^{n}|2\sin\pi r\omega|$
arises in dynamical systems as the magnitude $|z|$ in the skew product
$(\theta,z)\longmapsto(\theta+\omega,2z\sin\pi\theta)$ (with initial
condition $(\omega,1)$). However putting $z=\exp(2i\pi\omega)$ we
obtain the representation $P_{n}(\omega)=\prod_{r=1}^{n}\left|1-z^{r}\right|$
which is the modulus of the restricted Euler function, and links us
to partition theory (amongst other things). Further if we take the
$q-$Pochhammer symbol $(a;q)_{n}$ and put $a=q=z=\exp(2i\pi\omega)$
we have $P_{n}(\omega)=\left|(z;z)_{n}\right|$ which links us to
$q-$series and String Theory. Finally we have $\log P_{n}(\omega)=\sum_{1}^{n}f(r\omega)$,
where $f(x)=\frac{1}{2}\log(2-2\cos2\pi x)$. The latter function
is the harmonic conjugate of the sawtooth function $\pi\left(\{x\}-\frac{1}{2}\right)$
(see \cite{Knill2012}). The sum $\sum_{1}^{n}f(r\omega)$ is the
Birkhoff sum of $f$, and links us to ergodic theory and KAM theory.

\section{\label{sec:Survey}A brief survey of key results }

It seems that the function $P_{n}(\omega)$ has to date been studied
relatively independently in at least several of the fields noted above,
no doubt partly as a result of the diverse terminology and representations
in play. It seems worthwhile drawing together the various strands
of study, and we provide a brief survey below.

\subsection{Initial remarks}

The situation is very simple if $\omega$ is rational. Let $\omega=p/q$
where $p/q$ is in its lowest terms, then $P_{n}(\omega)=0$ for $n\ge q$.
In addition there is an old and rather elegant result we shall make
heavy use of, namely that $\prod_{r=1}^{q-1}2\sin\pi rp/q=q$, which
becomes in our notation $P_{q-1}(p/q)=q$. We have been unable to
attribute this result, but it is crucial to our estimates for irrational
$\omega$, and so we provide a simple proof in Appendix \ref{sec:SineAppendix}. 

When $\omega$ is irrational, $P_{n}(\omega)$ is never $0$. The
contrast with the case of $\omega$ rational suggests that the number
theoretic properties of $\omega$ are important, which is indeed confirmed
by later results. In addition, if for some $\omega_{0}$ we have $\lim\sup P_{n}(\omega_{0})>0$,
then since $\lim P_{n}(\omega)=0$ for rational $\omega$ arbitrarily
close to $\omega_{0}$, the behaviour of the sequences $P_{n}(\omega)$
are highly sensitive to the value of $\omega$ around $\omega_{0}$
- a harbinger of chaotic behaviour in any Dynamical Systems incorporating
such orbits.

\subsection{Growth with $n$ of the norm $\left\Vert P_{n}(\omega)\right\Vert =\sup_{\omega}\left|P_{n}(\omega)\right|$}

The first in-depth study of the function $P_{n}(\omega)$ seems to
have been made by Sudler in 1964\cite{Sudler1964}, although Erd\foreignlanguage{english}{\H{o}}s
\& Szekeres previously stated a ``very easy'' result (without proof)
in 1959\cite{Erdos1959}%
\footnote{Their claim was that $\lim_{n\rightarrow\infty}\left\Vert P_{n}(\omega)\right\Vert ^{1/n}$
exists and lies between $1$ and $2$. Sudler found the limit precisely. %
}. Sudler%
\footnote{Freiman and Halberstam (1988) attribute this result to Wright\cite{Freiman1988}
but from a careful reading of both papers \cite{Sudler1964,Wright1964}
it seems that Sudler has priority. Sudler does however acknowledge
the help of Wright as a referee in improving the proofs, and Wright
also improves Sudler's result in his own subsequent paper.%
} showed that in the limit the norm grows exponentially with $n$,~
\footnote{More precisely he showed $\left\Vert P_{n}(\omega)\right\Vert ^{1/n}=E+O(\log n/n)$
where $E$ is the constant above%
} with the growth rate $E$ being given by the formula: 
\[
E=\lim_{n\rightarrow\infty}\left\Vert P_{n}(\omega)\right\Vert ^{1/n}=\omega_{0}^{-1}\int_{0}^{\omega_{0}}\log\left|2\sin\pi\omega\right|d\omega=1.2197...
\]
where $\omega_{0}$ is the (unique) solution in $[1/2,1]$ of $\int_{0}^{\omega_{0}}\omega\cot\pi\omega\, d\omega=0$.
Further he also showed that $\left\Vert P_{n}(\omega)\right\Vert $
is achieved at $\omega_{n}$ where $\omega_{n}\sim\omega_{0}/n\in[1/2n,1/n]$
as $n$ grows. 

Freiman and Halberstam (1988) \cite{Freiman1988} later provided an
alternate proof which gives the same result in the even more elegant
form $E=2\sin\pi\omega_{0}$ (where $\omega_{0}$ is as above). (Incidentally
this seems to be the first paper to treat the function $P_{n}(\omega)$
as a first class citizen, ie as worthy of study in its own right rather
than as as a stepping stone to the estimation of other functions). 

In 1998 Bell et al \cite{Bell1998} proved a number of stronger results,
in particular that the norm of the decimated sub-product $\left\Vert \prod_{1}^{n}2\sin r^{k}\omega\right\Vert $
grows exponentially for any $k\ge1$. More recently Jordan Bell (2013)
\cite{Bell2013} adapted the method of Wright \cite{Wright1964} to
show $\left\Vert P_{n}(\omega)\right\Vert \sim C_{1}\sqrt{n}E^{n}$,
and also generalised the result to the $L_{p}$ norm: $\left\Vert P_{n}(\omega)\right\Vert _{p}=\left(\int_{0}^{1}P_{n}(\omega)^{p}d\omega\right)^{1/p}\sim C_{1}\left(C_{2}n^{-3/2}\right)^{1/p}\sqrt{n}E^{n}$
$ $ for calculated constants $C_{1},C_{2}.\,$%
\footnote{$C_{2}$ is actually $O(p^{-1/2})$, but is independent of $n$.%
}

\subsection{Growth of peaks of the sequence $P_{n}(\omega)$ at fixed $\omega$ }

We might expect that as the norm of the function $P_{n}(\omega)$
grows exponentially, then the pointwise growth rate (the growth rate
of the sequence $P_{n}=P_{n}(\omega)$ at a fixed value of $\omega$)
would also be exponential. However this turns out not to be the case.
Using the theory of uniform distribution, Lubinsky \cite{LubinskyPade1987}
showed that for almost all $\omega$, $\lim_{n\rightarrow\infty}P_{n}^{1/n}=1$,
ie the growth is ae sub-exponential, not exponential. This apparent
conflict is explained by Figure \ref{fig:PnNorm} in which we see
that the exponential growth of the norm is achieved at a peak which
is uncharacteristic of the rest of the function. This peak narrows
and converges on 0 as $n$ grows, so that for any fixed value of $\omega$
the peak will pass it for some value of $n$, after which the growth
at that point will revert to being sub-exponential.

\begin{figure}
\noindent \begin{centering}
\includegraphics[scale=0.5]{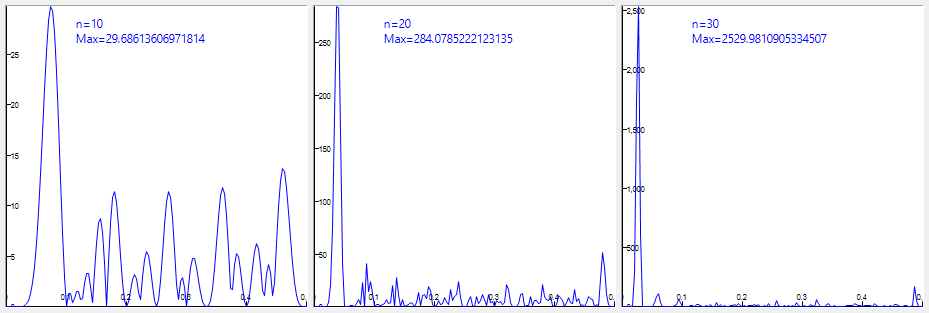}
\par\end{centering}

\protect\caption{\label{fig:PnNorm}$P_{n}(\omega)$ plotted over the interval $[0.0.5]$
for various values of $n$. Note how the norm is achieved at a point
which is converging on the origin, but that the exponential growth
achieved here is uncharacteristic.}
\end{figure}

In \cite{Knill2012} Knill \& Lesieutre adapted Herman's Denjoy-Koksma
result \cite{HermanDenjoy1979} to show for some constant $C$ that
$P_{n}(\omega)<n^{Cn^{1-1/s}\log n}$ when $\omega$ is of Diophantine
type $s\ge1$.~%
\footnote{ie There is some constant $c>0$ such that $|\omega-p/q|>c/q^{1+s}$
for any rational $p/q$ in lowest terms%
}

Lubinsky in a later paper \cite{Lubinsky1998} studied the problem
in the context of $q$-series and showed that, for almost all $\omega$,
and all $\epsilon>0$, there are constants $N,C$ (dependent on $\omega,\epsilon$)
such that for all $n>N$ we have $P_{n}(\omega)\le n^{C(\log n)^{\epsilon}}$,
and further, when $\omega$ is of Diophantine type 1, that $P_{n}(\omega)\le n^{C}$
. Amongst other results he also showed that, for all irrational $\omega$,
$\limsup_{n}P_{n}(\omega)\ge n$ (from which we deduce $C\ge1$ above),
and conjectured that for all $\omega$, $\liminf P_{n}(\omega)=0$.
He established that the latter result certainly holds for $\omega$
with unbounded partial quotients.

\subsection{Growth of the sequence $P_{n}(\omega)$ when $\omega$ is the golden
rotation $(\sqrt{5}-1)/2$}

The results of the previous section bound the peak growth of $P_{n}(\omega)$.
However the peaks of this function are very different from its values
elsewhere (see Figure \ref{fig:Pnn}). Certain applications, in particular
in the study of strange non-chaotic attractors (SNA's), require a
sharper estimate of the size of $P_{n}(\omega)$ at every point $n$,
and not just at the peaks. 

Working in the context of KAM theory, Knill and Tangerman studied
the Birkhoff sum representation $S_{n}(\omega)=\sum_{r=1}^{n}\log(2-2\cos2\pi r\omega)$
in their 2011 paper \cite{Knill2011}. It is easy to show that  $S_{n}(\omega)=2\log P_{n}(\omega)$.
As is often the case with this type of problem, they chose to study
the ``simplest'' irrational number, the golden mean (or more accurately,
its fractional part $\omega=(\sqrt{5}-1)/2$) which is of Diophantine
type 1 and has rational convergents $F_{n-1}/F_{n}$ where $F_{n}$
is the $n$th Fibonacci number (indexed from $F_{0}=1$). Taking the
sequence $(F_{n})$ as a renormalisation scale, they presented experimental
graphical and numerical evidence for the existence of an asymptotic
renormalisation function. The renormalisation approach was also earlier
studied by Kuznetsov et al (1995) \cite{Kuznetsov1995a} in a slightly
more general setting, where they used polynomial approximation to
obtain strong numerical evidence also for asymptotic renormalisation
functions.

Assuming the existence of this asymptotic function as a hypothesis,
Knill and Tangerman deduced the following consequences:
\begin{thm}
Consequences of the hypothesis\end{thm}
\begin{enumerate}
\item The log average of the Birkhoff sum tends to a constant along the
renormalisation subsequence, ie $S_{F_{n}}(\omega)/\log F_{n}\longrightarrow c$
for some constant $c$.
\item The sequence $S_{n}(\omega)/\log n$ has accumulation points at $0$
and $2$. 
\item The sequence $S_{n}(\omega)/\log n$ is bounded. %
\footnote{Lubinsky \cite{Lubinsky1998} proved this result for almost all $\omega$
(in the $q-$Pochhammer form $\log|(q;q)_{n}|=O(\log n)$), but not
necessarily for the specific value of the golden mean %
} 
\end{enumerate}
In this paper we will give a rigorous proof of the (slightly stronger)
analogous results which are set out below in terms of $P_{n}(\omega)$.
But $S_{n}(\omega)=\sum_{r=1}^{n}\log(4\sin^{2}2\pi r\omega)=2\log P_{n}(\omega)$
and from this it is easily seen that the results below imply the results
above. (Note that (2) above is the result of combining (1),(2) below).
\begin{thm}
\label{prop:K=000026T}The following results hold:\end{thm}
\begin{enumerate}
\item For some constant $c$, $P_{F_{n}}(\omega)\longrightarrow c$.
\item For the same constant $c$, $P_{F_{n}-1}(\omega)/F_{n}\longrightarrow$$c/2\pi\sqrt{5}$.
\item There are real constants $C_{1}\le0<1\le C_{2}$ such that $n^{C_{1}}\le P_{n}(\omega)\le n^{C_{2}}$.
\end{enumerate}
The proof of the first foundational result, namely that $P_{F_{n}}(\omega)\longrightarrow c$
for some constant $c$, will occupy the bulk of the paper. In section
\ref{sec:Consequences} we will deduce results (2) and (3).

\section{\label{sec:Overview}Statement of main result \& Overview of Proof}

At the end of the previous section we described how Knill \& Tangerman
recently presented experimental graphical and numerical evidence for
the existence of an asymptotic renormalisation function when $\omega=\left(\sqrt{5}-1\right)/2$.
From this they deduced three consequences. However we will show in
Section \ref{sec:Consequences} that the second and third consequences
flow directly from the first, and have no dependency on the experimental
function. Our main contribution in this paper is to establish the
first consequence rigorously without reference to the experimental
function, that is, we will prove that the sequence $P_{F_{n}}(\omega)$
converges to a constant.

This rather simple statement belies the surprising amount of work
which seems necessary to prove it. However it is worth noting that
both Knill and Lubinsky remark that this is one problem area where
established procedures and powerful tools fall short. This has also
been our own experience, and we have felt very much forced back to
a proof from first principles. 

 Following renormalisation terminology, we ``decimate'' the sequence
$P_{n}=\left|\prod_{r=1}^{n}2\sin\pi r\omega\right|$ by picking every
$F_{n}$th element to yield a ``renormalisation sub-sequence'' $Q_{n}=\left|\prod_{r=1}^{F_{n}}2\sin\pi r\omega\right|$.
Our main result is now the following:
\begin{thm}
The sequence $Q_{n}=\left|\prod_{r=1}^{F_{n}}2\sin\pi r\omega\right|$
is convergent to a constant $c=2.407....$
\end{thm}
The proof of this theorem will occupy the main body of the paper (sections
\ref{sec:Overview}--\ref{sec:EndMain}). 

In Section \ref{sec:Consequences} we deduce from the main result
two other results reported by Knill \& Tangerman. In particular this
includes the result that the Sudler product growth at $\omega$ is
bounded by a power law. Knill \& Tangerman suggested that this particular
result would flow from a modification of the proof of the Denjoy-Koksma
result in ergodic theory, but on closer examination further work appears
necessary. We have again found the need to derive this corollary from
first principles.

\subsection{\label{sub:Overview}Overview of the proof of the main result (sections
\ref{sec:Preliminaries}--\ref{sec:EndMain}):}

In section \ref{sec:Preliminaries}, together with some other preliminaries,
we separate out the proofs of a number of ancillary results from the
overall flow, in an attempt to make clearer the main lines of reasoning
in the other sections.

In section \ref{sec:MainStart} we introduce a core strategy which
is to exploit the continued fraction convergents to the inverse golden
mean $\omega$. These convergents are the ratios of subsequent Fibonacci
numbers $F_{n-1}/F_{n}$, and $\omega=\left(F_{n-1}-(-\omega)^{n}\right)/F_{n}$
(see \eqref{eq:wEstimate}). This gives us:
\[
Q_{n}=\left|\prod_{r=1}^{F_{n}}2\sin\pi r\omega\right|=\left|\prod_{r=1}^{F_{n}}2\sin\pi r\left(F_{n-1}-(-\omega)^{n}\right)/F_{n}\right|
\]

This allows us to develop a representation of $Q_{n}$ as a product
of three rather more tractable products, namely: 
\[
Q_{n}=A_{n}\, B_{n}\, C_{n}=\left(2F_{n}\sin\pi\omega^{n}\right)\left(\prod_{t=1}^{F_{n}-1}\frac{s_{nt}}{2\sin\pi\frac{t}{F_{n}}}\right)\,\prod_{t=1}^{F_{n}/2}\left(1-\frac{s_{n0}^{2}}{s_{nt}^{2}}\right)
\]
where $s_{nt}=2\sin\pi\left(t/F_{n}-\omega^{n}\left([F_{n-1}t]/F_{n}-1/2\right)\right)$.

It is easy to show that $A_{n}\rightarrow2\pi/\sqrt{5}$, and in fact
the products $B_{n},C_{n}$ also converge to strictly positive limits.
However the latter demonstrations require significantly greater effort,
and receive their own sections.

In section \ref{sec:Cn} we shall deal with the convergence of the
simpler of the two products, namely $C_{n}=\prod_{t=1}^{F_{n}/2}\left(1-\frac{s_{n0}^{2}}{s_{nt}^{2}}\right)$.
In section \ref{sec:EndMain} we shall deal with the convergence of
$B_{n}=\left(\prod_{t=1}^{F_{n}-1}\frac{s_{nt}}{2\sin\pi\frac{t}{F_{n}}}\right)$.
This requires the most work and is broken down into several significant
sub-sections.

\section{\label{sec:Preliminaries}Preliminaries}

\subsection{Notation}

We will make use of both modulo arithmetic and floor functions. Since
the box notation $[.]$ is often used for both purposes, in this paper
we will use the following conventions:
\begin{itemize}
\item For a given positive modulus $q\ge1$, we use $[r]$ to represent
the residue of $r\bmod q$ in the residue set $\{0,\ldots,\, q-1\}$,
for example $[-1]=q-1$.
\item We use $\left\lfloor x\right\rfloor $ to represent the floor of $x$,
ie the largest integer less than or equal to $x$, for example $\left\lfloor -0.5\right\rfloor =-1$.
\end{itemize}
We also make extensive use of the following ``almost standard''
notation, which we define here precisely in order to eliminate any
ambiguity over edge cases:
\begin{itemize}
\item The fractional part function $\left\{ x\right\} $ maps $x$ to $x-\left\lfloor x\right\rfloor \in[0,1)$,
for example $\{-1.25\}=0.75$
\item $f(x)=O(g(x))$ as $x\rightarrow C\in[-\infty,+\infty]$ means that
there is a positive real constant $M$ and a neighbourhood $N(C)$
such that $|f(x)|<M|g(x)|$ for $x\in N(C)$, for example $1/(c-x)=O(1/x)$
as $x\rightarrow+\infty$. Normally $C$ is $0$ or $+\infty$, and
will be omitted if clear from the context. 
\end{itemize}

\subsubsection{\label{sub:GenSummation}Generalised notation for sums and products }

As usual we will define the empty sum to have the value 0, and the
empty product to have the value $1$. 

Given a summable sequence $(a_{r})$ we will find it useful to define
a generalised summation notation $\sum_{r=x}^{y}a_{r}$ to include
real (rather than integer) upper and lower bounds $x,y$. We do this
by defining the step function $f(t)=a_{r}$for $t\in[r,r+1)$, and
then $\sum_{r=x}^{y}a_{r}=\int_{x}^{y}f(t)\mathrm{dt}$. If $f(t)>0$
on $[x,y)$ we define the multiplicative analogue as $\prod_{r=x}^{y}a_{r}=\exp\int_{x}^{y}\log f(t)\mathrm{dt}$.
For example for odd integers $n=2k+1$: 
\[
\sum_{1}^{n/2}a_{r}=\sum_{1}^{k}a_{r}+\frac{1}{2}a_{k+1}
\]
Note that for integer $x,y$ the definitions coincide with normal
summation and product notation.

\subsection{Special sequences used in this paper}

In addition to the Fibonacci sequence $\left(F_{i}\right)=(0,1,1,2,3,5,8...)$,
we make extensive use of a number of derived sequences which we define
here for convenience. Note we only define them for integer $n,t$
and $n\ge1$.

\begin{eqnarray}
s_{nt} & = & 2\sin\pi\left(\frac{t}{F_{n}}-\omega^{n}\left(\frac{\left[tF_{n-1}\right]}{F_{n}}-1/2\right)\right)\label{eq:snt}\\
\xi_{nt} & = & \begin{cases}
\frac{\left[tF_{n-1}\right]}{F_{n}}-\frac{1}{2} & \left([t]\ne0\bmod F_{n}\right)\\
0 & \left([t]=0\bmod F_{n}\right)
\end{cases}\\
\xi_{\infty t} & = & \{t\omega\}-\frac{1}{2}\\
h_{nt} & = & \cot\frac{\pi t}{F_{n}}\sin\pi\omega^{n}\xi_{nt}\;([t]\ne0\bmod F_{n})
\end{eqnarray}

Note that $s_{nt}=2\sin\pi\left(t/F_{n}-\omega^{n}\xi_{nt}\right)$
when $[t]\ne0\bmod F_{n}$, but not when $[t]=0\bmod F_{n}$ due to
the alternative definition of $\xi_{nt}$. This reflects the fact
that the two sequences play very different roles, and each definition
makes sense in its own context. We have also chosen to leave $h_{nt}$
undefined for $[t]=0\bmod F_{n}$.
\begin{lem}
\label{lem:baseSeq}For the sequences $s_{nt},\xi_{nt},\xi_{\infty t}$
defined above:
\begin{enumerate}
\item For fixed $n\ge1$, the sequences $\left|s_{nt}\right|,\xi_{nt},h_{nt}$
are periodic sequences of period $F_{n}$, and further $s_{nt},\xi_{nt}$
are both odd sequences in $t$ (ie of the form $a_{t}=-a_{-t}$) and
$h_{nt}$ is an even sequence in $t$ (ie of the form $a_{t}=a_{-t})$.
\item Both $\left|\xi_{nt}\right|<1/2$ and $\left|\xi_{\infty t}\right|<1/2$
with the exception of $\xi_{\omega0}=-1/2$. 
\item In the range $0\le t\le F_{n}-1$, $s_{nt}\ge s_{n0}>0$ with equality
only at $t=0$. For any $t$, $s_{n,F_{n}+t}=-s_{nt}$
\item $\xi_{n,F_{n}-t}=-\xi_{nt}$ whereas $s_{n,F_{n}-t}=s_{nt}$, and
$h_{n,F_{n}-t}=h_{nt}$
\item For $1\le t\le F_{n-1}$ we have $\xi_{nt}=\xi_{\infty t}+O(\omega^{n})$
and $\lim_{n\rightarrow\infty}\xi_{nt}=\xi_{\infty t}$
\end{enumerate}
\end{lem}
\begin{proof}
~\end{proof}
\begin{enumerate}
\item Note that $\left\{ tF_{n-1}/F_{n}\right\} $ is of period $F_{n}$,
and the periodicity results follow, noting also that $\left|\sin\pi x\right|$
is of period $1$. Also we have $\left\{ -x\right\} =1-\{x\}$, from
which the oddness of $\xi_{nt}$ immediately follows. The oddness
of $s_{nt}$ then follows from the oddness of $\sin x$. The evenness
of $h_{nt}$ follows from the oddness of both $\cot$ and $\sin$.
\item Both results follow from $0<\{x\}<1$ unless $x=0$. But $\{tF_{n-1}/F_{n}\}=0$
only for $[t]=0\bmod F_{n}$ and then $\xi_{nt}=0$. And $\{t\omega\}=0$
only for $t=0$.
\item For $t=0$ we have $s_{nt}=s_{n0}=2\sin\pi\omega^{n}/2>0$. For $n=1,2$
the only possibility is $t=0$, but for $n\ge3$ and $1\le t\le F_{n}-1$,
then $s_{nt}\ge s_{n1}=2\sin\pi\left(F_{n}^{-1}-\omega^{n}\xi_{nt}\right).$
But $\left|\xi_{nt}\right|<1/2$, and $F_{n}^{-1}=\sqrt{5}\omega^{n}/(1-(-1)^{n}\omega^{2n})>2\omega^{n}$
so that $s_{n1}>s_{n0}>0$. The second part follows by noting that
substituting $F_{n}+t$ in $s_{nt}$ simply adds $\pi$ to the argument
of the sine function.
\item These now follow easily from the previous results
\item Since $t\ne0$, we have $\xi_{nt}-\xi_{wt}=\{tF_{n-1}/F_{n}\}-\{t\omega\}$.
Now by \eqref{eq:w} $t\omega=tF_{n-1}/F_{n}-t(-\omega^{n})/F_{n}$,
but $t<F_{n}$ so $\left|t\omega-tF_{n-1}/F_{n}\right|<\omega^{n}<1/F_{n}$
which means $\{t\omega\}$ is always inside the interval $\{tF_{n-1}/F_{n}\}\pm1/F_{n}$,
and we can deduce that $\left|\xi_{nt}-\xi_{wt}\right|<\omega^{n}$.
The results follow.
\end{enumerate}

\subsection{Inequalities}

We gather here various inequalities which we will need during the
main proofs.
\begin{lem}
\label{lem:ConvexSine}For $x$ in $(0,\pi/2)$ we have $2x/\pi<\sin x<x$
\end{lem}
\begin{proof}
The derivative of $f(x)=x-\sin x$ is $1-\cos x$ which is positive.
So $f(x)$ is increasing, and $f(0)=0$, and the right side inequality
follows. For the left side we use the fact that $\sin x$ is convex
in this interval and hence lies above the line segment joining $(0,0)$
and $(\pi/2,1)$. But this is $2x/\pi$. \end{proof}
\begin{lem}
\label{lem:ProdBounds}For $n\ge2$, let $(a_{t})$ $t=1...n$ be
a sequence of real numbers satisfying $\left|a_{t}\right|<1$ with
$A=\sum\left|a_{t}\right|<1$. Then
\[
1-A<\prod_{t=1}^{n}(1+a_{n})<\frac{1}{1-A}
\]
\end{lem}
\begin{proof}
$\prod_{t=1}^{n}(1+a_{t})\ge\prod_{t=1}^{n}(1-\left|a_{t}\right|)$.
Then $\prod_{t=1}^{n}(1-\left|a_{t}\right|)>1-A$ is clearly true
for $n=2$ and the left hand side of the result follows by induction.

Similarly $\prod_{t=1}^{n}(1+a_{t})\le\prod_{t=1}^{n}(1+\left|a_{t}\right|)$.
Then $\prod_{t=1}^{n}(1+\left|a_{t}\right|)<\sum_{r=0}^{n}\left(\sum\left|a_{t}\right|\right)^{r}=(1-A^{n+1})/(1-A)$
proving the right hand side of the result.
\end{proof}

\subsection{Results on Fibonacci numbers}

We will need various standard results concerning the Fibonacci sequence
$\left(F_{i}\right)=\left(0,1,1,2,3,5,8...\right)$, and these are
summarised without proof in Appendix B. We give here some other simple
results we will need later.
\begin{lem}
\label{lem:Fn-1Inverse}For $n\ge1$, defining $[0]^{-1}=[0]\bmod1$,
the inverse of $F_{n-1}\bmod F_{n}$ exists and is $\left[(-1)^{n}F_{n-1}\right]$\end{lem}
\begin{proof}
From \eqref{eq:FibProduct} we have $F_{n+1}F_{n-1}-F_{n}^{2}=(-1)^{n}$
and substituting the definition $F_{n+1}=F_{n}+F_{n-1}$ gives $F_{n-1}^{2}\equiv(-1)^{n}\bmod F_{n}$,
whence $F_{n-1}.(-1)^{n}F_{n-1}\equiv1\bmod F_{n}$.
\end{proof}

\subsubsection{\label{sub:Fibonacci-representation}Representation by Fibonacci
numbers}

We will use Zeckendorf's representation of $k\ge1$ as a sum of Fibonacci
numbers via the following definition:%
\footnote{We give here our own recursive definition, as this obviates the need
to prove existence or uniqueness of the resulting representation.
It also provides for straightforward translation into modern programming
languages, and is easily extended to an Ostrowski representation.%
}
\begin{defn}
\label{def:FibSumLength} Fibonacci summation algorithm for the Zeckendorf
representation

For $n\ge0$, define the \emph{Fibonacci floor} $F(n)$ to be the
largest Fibonacci number $F_{i}\le n$, for example $F(0)=F_{0},\, F(1)=F_{2},\, F(7)=F_{5}.$

We now define the \emph{Fibonacci sum} $\sum^{F}(n)$ to be the series
defined recursively by $\sum^{F}(n)=F(n)+\sum^{F}(n-F(n)),\;\sum^{F}(0)=F_{0}$.
For example $\sum_{F}(7)=F_{5}+F_{3}+F_{0}=5+2+0$. 

We define the Fibonacci length $F_{L}(n)$ of $n$ to be the length
of the Fibonacci sum ignoring the $F_{0}$ element. For example the
Fibonacci lengths of $0,7$ are $0,2$ respectively.
\end{defn}
Since algorithm is deterministic, it provides a unique representation
for each $n$. Further for $i\ge2$ we have $F_{i}\le n<F_{i+1}$,
and $F_{i+1}=F_{i}+F_{i-1}$, we have $n-F(n)<F_{i+1}-F_{i}=F_{i-1}$,
so that if $F_{i}$ is in the representation, $F_{i-1}$ is not. Finally
note that since $F_{2}=F_{1}=1$, we never have $F(n)=F_{1}$. 
\begin{defn}
The Binary Fibonacci representation $n=\sum_{s=1}^{m}b_{s}F_{s}$
for $n\ge1$ 

We will find it useful to translate the sum $\sum^{F}(n)$ into the
equivalent representation $\sum_{s=1}^{m}b_{s}F_{s}$ where $b_{s}=1$
if $F_{s}$ is in $\sum^{F}(n),$and $b_{s}=0$ otherwise, and $F_{m}=F(n)$.
\end{defn}
Note that the Fibonacci length $F_{L}(n)$ of $n$ is now $\sum_{s=1}^{m}b_{s}\le m-1$
since $F_{1}$ is never in $\sum^{F}(n)$.
\begin{lem}
\label{lem:FLest} If $n\ge1$ has the representation $\sum_{s=1}^{m}b_{s}F_{s}$
then: 
\begin{eqnarray}
m & \le & \left\lfloor \left(\log n\,+1\right)/\log(1+\omega)\right\rfloor \\
F_{L}(n) & \le & \left\lfloor (\log n\,+1)/\log(2+\omega)\right\rfloor 
\end{eqnarray}
\end{lem}
\begin{proof}
Since $F_{m}=F(n)\le n$, we use from \eqref{eq:Fn=00003D} $F_{m}=(\omega^{-m}-(-\omega)^{m})/\sqrt{5}$
and we deduce (using $\log(1+x)<x$, and $\omega^{-1}=1+\omega)$):
\begin{eqnarray}
m & = & \max\{j:\omega^{-j}\le\sqrt{5}n+(-\omega)^{j}\}\label{eq:kEst-1}\\
 & = & \max\{j:j\le\log\left(\sqrt{5}n+(-\omega)^{j}\right)/\log(1+\omega)\nonumber \\
 & \le & \left\lfloor \left(\log n\,+1\right)/\log(1+\omega)\right\rfloor \nonumber 
\end{eqnarray}

Now the Fibonacci length $F_{L}(n)=\sum_{s=1}^{m}b_{s}\le m-1$, but
in fact we can do better than this. Since $\sum^{F}(n)$ contains
no two consecutive $F_{i}$ we have: 
\begin{equation}
F_{L}(n)\le\left\lfloor m/2\right\rfloor \label{eq:FLestimate}
\end{equation}

The result follows using $(1+\omega)^{2}=2+\omega$.
\end{proof}

\section{\label{sec:MainStart}The Decomposition $Q_{n}=A_{n}B_{n}C_{n}$}

As described in section \ref{sub:Overview}, we develop a decomposition
of $Q_{n}$ into a product of three other products, each of which
converges to a positive constant. We shall prove the convergence of
the first of these products within this section (as it is very straightforward),
and the other two we shall deal with in subsequent sections.

Our central motivation here is to substitute the Fibonacci identity
$\omega=\left(F_{n-1}/F_{n}\right)-(-\omega)^{n}/F_{n}$ (see \eqref{eq:wEstimate})
into the definition of $Q_{n}$ and hence express $\left|\prod2\sin\pi r\omega\right|$
as a perturbation of the rational sine product $\left|\prod2\sin\pi r(F_{n-1}/F_{n})\right|$,
the latter product being equal to $F_{n}$ (see \eqref{eq:lemSinProduct}).
This reduces the problem to one of demonstrating that the perturbation
function itself has suitable behaviour, and this proves equivalent
to showing that the product $B_{n}C_{n}$ converges as $n$ grows.
However rather than treating $B_{n}C_{n}$ as a single product, it
is simpler to prove separately that each of $B_{n}$ and $C_{n}$
converge.

The substitution above gives us $Q_{n}=\left|\prod2\sin\pi r\left((F_{n-1}/F_{n})-(-\omega)^{n}/F_{n}\right)\right|$
which is a perturbation of the the argument in each term of $\left|\prod2\sin\pi r(F_{n-1}/F_{n})\right|$
by a delta of $-r(-\omega)^{n}/F_{n}$. The sum of these deltas is
non-zero, but some of the techniques we shall use to prove the convergence
of $B_{n},C_{n}$ require that the sum of the deltas is 0. Fortunately,
as we shall see, we can fix this by re-basing the arguments to result
in a delta of $\omega^{n}(r/F_{n}-1/2)$ - which then provides a zero
sum for the deltas. This is most economically achieved once and for
all at the beginning of our proof, and will simplify later proofs
at the cost introducing a non-intuitive first step below. However
once done, we proceed to make the substitution for $\omega$, and
the decomposition then follows naturally.

\begin{lem}
\label{lem:QnFactored}For $n\ge1$ and $s_{nt}=2\sin\pi\left(t/F_{n}-\omega^{n}\left(\frac{[F_{n-1}t]}{F_{n}}-\frac{1}{2}\right)\right)$
we have $Q_{n}=\left|\prod_{r=1}^{F_{n}}\Bigl(2\sin\pi r\omega\Bigr)\right|=A_{n}\, B_{n}\, C_{n}$
where:
\end{lem}
\begin{eqnarray}
A_{n} & = & 2F_{n}\sin\pi\omega^{n}\rightarrow\frac{2\pi}{\sqrt{5}}\label{eq:QnDecomp}\\
B_{n} & = & \left(\prod_{t=1}^{F_{n}-1}\frac{s_{nt}}{2\sin\pi\frac{t}{F_{n}}}\right)\\
C_{n} & = & \prod_{t=1}^{F_{n}/2}\left(1-\frac{s_{n0}^{2}}{s_{nt}^{2}}\right)
\end{eqnarray}

We first deal with the convergence of $A_{n}$ by observing that since
$\omega<1$, we have $A_{n}=2F_{n}\sin\pi\omega^{n}\sim2F_{n}\pi\omega^{n}$
and the result follows by \eqref{eq:Fnw}.

We start the main proof by carrying out the step discussed above to
re-base our perturbation deltas. First we exploit the symmetry of
the sine function around $\pi/2$, observing that a change of variables
$r\mapsto F_{n}-r$ gives us $\prod_{r=1}^{F_{n}-1}\Bigl(2\sin\pi r\omega\Bigr)=\prod_{r=1}^{F_{n}-1}\left(2\sin\pi(F_{n}-r)\omega\right)$
and hence for any $n\ge1$, using the product of sines formula:

\begin{eqnarray}
Q_{n}^{2} & = & \left(2\sin\pi F_{n}\omega\right)^{2}\prod_{r=1}^{F_{n}-1}\Bigl(2\sin\pi r\omega\Bigr)\left(2\sin\pi(F_{n}-r)\omega\right)\nonumber \\
 & = & \left(2\sin\pi F_{n}\omega\right)^{2}\prod_{r=1}^{F_{n}-1}2\left(\cos\pi(F_{n}-2r)\omega-\cos\pi F_{n}\omega\right)
\end{eqnarray}

We can now use identity \eqref{eq:Fnw} and the cosine double angle
formula to obtain:
\begin{eqnarray}
Q_{n}^{2} & = & \left(2\sin\pi\omega^{n}\right)^{2}\prod_{r=1}^{F_{n}-1}2(-1)^{F_{n-1}}\left(\cos\pi((-\omega)^{n}+2r\omega)-\cos\pi(-\omega)^{n}\right)\nonumber \\
 & = & \left(2\sin\pi\omega^{n}\right)^{2}(-1)^{(F_{n}-1)F_{n-1}}\prod_{r=1}^{F_{n}-1}4\left(\sin^{2}\frac{1}{2}\pi\omega^{n}-\sin^{2}\pi(r\omega+\frac{1}{2}(-\omega)^{n})\right)\nonumber \\
 & = & \left(2\sin\pi\omega^{n}\right)^{2}(-1)^{(F_{n}-1)(F_{n-1}+1)}\prod_{r=1}^{F_{n}-1}4\left(\sin^{2}\pi(r\omega+\frac{1}{2}(-\omega)^{n})-\sin^{2}\frac{1}{2}\pi\omega^{n}\right)
\end{eqnarray}

Now if $F_{n}$ is odd then $F_{n}-1$ is even, and if $F_{n}$ is
even then by \eqref{eq:FnEven} $F_{n-1}+1$ is even, and so for any
$n$ we have $(-1)^{(F_{n}-1)(F_{n-1}+1)}=1$. We have therefore shown
that
\begin{eqnarray}
Q_{n}^{2} & = & \left(2\sin\pi\omega^{n}\right)^{2}\prod_{r=1}^{F_{n}-1}4\left(\sin^{2}\pi(r\omega+\frac{1}{2}(-\omega)^{n})-\sin^{2}\frac{1}{2}\pi\omega^{n}\right)\label{eq:QnEst2}
\end{eqnarray}

This completes the re-basing step. We are now ready to develop the
expression for $Q_{n}$ as a perturbation of the rational sine product
$\left|\prod2\sin\pi r(F_{n-1}/F_{n})\right|$.

The product above is empty for $n=1,2$. For $n\ge3$ we develop the
second sine term above by substituting the Fibonacci identity and
then using \eqref{eq:Fnw} to obtain

\begin{equation}
\sin\pi(r\omega+(-\omega)^{n}/2)=\sin\pi\left(\frac{rF_{n-1}}{F_{n}}-(-\omega)^{n}(\frac{r}{F_{n}}-\frac{1}{2})\right)
\end{equation}

Substituting the residue $t=[rF_{n-1}]$ in the right hand term and
using Lemma \ref{lem:Fn-1Inverse} we obtain 
\begin{equation}
\sin\pi(r\omega+(-\omega)^{n}/2)=\pm\sin\pi\left(\frac{t}{F_{n}}-(-\omega)^{n}(\frac{[(-1)^{n}F_{n-1}t]}{F_{n}}-\frac{1}{2})\right)
\end{equation}

Now observe that $x\mapsto\{x\}-1/2$ is an odd function (for non-integer
$x$), and we use this fact to simplify the right side to obtain finally
for every $1\le r\le F_{n}-1$  
\begin{eqnarray}
\Bigl|\sin\pi(r\omega+(-\omega)^{n}/2)\Bigr| & = & \Bigl|\sin\pi\left(\frac{t}{F_{n}}-\omega^{n}(\frac{[F_{n-1}t]}{F_{n}}-\frac{1}{2})\right)\Bigr|
\end{eqnarray}

Now the right hand side is $\Bigl|s_{nt}\Bigr|$, and for $1\le r\le F_{n}-1$
we also have $1\le t\le F_{n-1}$. In this range for $t$ we have
$s_{nt}>0$ by Lemma \ref{lem:baseSeq}. This gives us for $1\le s,t\le F_{n}-1$:
\begin{equation}
\Bigl|\sin\pi(r\omega+(-\omega)^{n}/2)\Bigr|=s_{nt}\label{eq:SineInequality}
\end{equation}

If we further observe that for $1\le r\le F_{n}-1$, $t=[rF_{n-1}]$
runs through a complete set of non-zero residues, so we can rewrite
\eqref{eq:QnEst2} for $n\ge1$ as:
\begin{eqnarray}
Q_{n}^{2} & = & \left(2\sin\pi\omega^{n}\right)^{2}\prod_{t=1}^{F_{n}-1}\left(s_{nt}^{2}-s_{n0}^{2}\right)\nonumber \\
 & = & \left(2\sin\pi\omega^{n}\right)^{2}\left(\prod_{t=1}^{F_{n}-1}s_{nt}\right)^{2}\prod_{t=1}^{F_{n}-1}\left(1-\frac{s_{n0}^{2}}{s_{nt}^{2}}\right)\label{eq:QnEst2-1}
\end{eqnarray}

We have almost proved Lemma \ref{lem:QnFactored}. To obtain the final
result, we deduce from \eqref{eq:lemSinProduct} that $\prod_{t=1}^{F_{n}-1}2\sin\pi\frac{t}{F_{n}}=F_{n}$
and the result follows (using $s_{nt}=s_{n(F_{n}-t)}$ from Lemma
\ref{lem:baseSeq}).

\section{\label{sec:Cn}The Convergence of $C_{n}=\prod_{t=1}^{F_{n}/2}\left(1-\frac{s_{n0}^{2}}{s_{nt}^{2}}\right)$}

In this step we show $C_{n}$ converges to a strictly positive constant.
This is not as straightforward as it appears at first sight as there
are terms in $s_{nt}$ which oscillate about $0$ but which are not
alternating. We therefore cannot assume that $C_{n}$ is decreasing.
Fortunately we are able to compare $C_{n}$ with a closely related
sequence which \emph{is} decreasing and therefore converges.

\begin{thm}
\label{thm:P3}The sequence $C_{n}=\prod_{t=1}^{F_{n}/2}\left(1-\frac{s_{n0}^{2}}{s_{nt}^{2}}\right)$
converges to 
\[
\prod_{t=1}^{\infty}\left(1-\frac{1}{20\left(t-\frac{1}{\sqrt{5}}\left(\{t\omega\}-\frac{1}{2}\right)\right)^{2}}\right)\simeq0.928
\]

\end{thm}
For $n\in\{0,1,2\}$the product defining $C_{n}$ is empty and $C_{n}=1$.
For the rest of this section we will assume $n\ge3$, and so by Lemma
\ref{lem:baseSeq} we have for $1\le t\le F_{n}-1$ that $s_{nt}>s_{n0}>0$

Hence we have $0<(1-s_{n0}^{2}/s_{nt}^{2})<1$ for every term in $C_{n}$,
and so $0<C_{n}<1$ for $n\ge3$. 

At this point we need to establish some estimates for the terms $s_{n0}/s_{nt}$.
First we develop some general estimates valid for all $0\le t<F_{n}$.

For $t=0$ we have 
\begin{equation}
s_{n0}=2\sin\pi(\omega^{n}/2)=\pi\omega^{n}(1+O(\omega^{2n}))\label{eq:sn0}
\end{equation}

For $1\le t\le F_{n}/2$, from \eqref{eq:Fnw} $F_{n}^{-1}=\sqrt{5}\omega^{n}(1+O(\omega^{2n}))$,
and from \eqref{eq:wEstimate}  $F_{n-1}/F_{n}=\omega+O(\omega^{2n})$
and so: 
\begin{eqnarray}
s_{nt} & = & 2\sin\pi\left(\left(t\sqrt{5}\omega^{n}(1+O(\omega^{2n})\right)-\omega^{n}\left(\{t\omega\}+tO(\omega^{2n})-\frac{1}{2}\right)\right)\nonumber \\
 & = & 2\sin\pi\omega^{n}t\left(\sqrt{5}-\frac{1}{t}\left(\{t\omega\}-\frac{1}{2}\right)+O(\omega^{2n})\right)\label{eq:sntGenEst}
\end{eqnarray}

Now let $q=\left\lceil \omega^{-3n/5}\right\rceil $%
\footnote{Here $3/5$ is chosen to optimise convergence, though other values
are possible. %
}. For $t\ge q$ we use $(\pi/2)\sin x>x$ (from Lemma \ref{lem:ConvexSine})
in \eqref{eq:sntGenEst} to give us for large enough $n$:
\begin{eqnarray}
\frac{s_{n0}}{s_{nt}} & < & \frac{\pi\omega^{n}(1+O(\omega^{2n}))}{(2/\pi)2\pi\omega^{n}t\left(\sqrt{5}-\frac{1}{t}\left(\{t\omega\}-\frac{1}{2}\right)+O(\omega^{2n})\right)}\nonumber \\
 & < & \frac{\pi(1+O(\omega^{2n}))}{4q\left(\sqrt{5}-q^{-1}\left(\{t\omega\}-\frac{1}{2}\right)+O(\omega^{2n})\right)}\nonumber \\
 & < & \frac{\pi(1+O(q^{-1}))}{4\sqrt{5}q}\nonumber \\
 & = & O(q^{-1})
\end{eqnarray}

Now choose $q\le q_{1}<q_{2}\le F_{n}/2$. We can now use from Lemma
\ref{lem:ProdBounds} $\prod(1-a_{n})>1-\sum\left|a_{n}\right|$ to
obtain:
\begin{eqnarray}
1>\prod_{t=q_{1}}^{q_{2}}\left(1-\frac{s_{n0}^{2}}{s_{nt}^{2}}\right) & > & \prod_{t=q_{1}}^{F_{n}/2}\left(1-\frac{s_{n0}^{2}}{s_{nt}^{2}}\right)\nonumber \\
 & > & 1-\sum_{t=q_{1}}^{F_{n}/2}O(q^{-2})>1-F_{n}O(\omega^{6n/5})\nonumber \\
 & = & 1-O(\omega^{n/5})
\end{eqnarray}

Now we consider the case of $t<q$. From \eqref{eq:sntGenEst} we
have $s_{nt}=2\sin\pi\omega^{n}t\left(\sqrt{5}-\left(\{t\omega\}-\frac{1}{2}\right)/t+O(\omega^{2n})\right)$,
and the largest term in the argument of the sine function is then
$O(\omega^{n}q)=O(\omega^{2n/5})$, so that for large enough $n$
we can make the argument as small as we like. So we can use $\sin x=x+O(x^{3})$
to give us:
\begin{eqnarray}
s_{nt} & = & 2\pi\omega^{n}t\left(\sqrt{5}-\frac{1}{t}\left(\{t\omega\}-\frac{1}{2}\right)+O(\omega^{2n})\right)+O(\omega^{6n/5})\nonumber \\
 & = & 2\sqrt{5}\pi\omega^{n}\left(t-\frac{1}{\sqrt{5}}\left(\{t\omega\}-\frac{1}{2}\right)+O(\omega^{n/5})\right)
\end{eqnarray}

We put $u_{t}=2\sqrt{5}\left(t-\frac{1}{\sqrt{5}}\left(\{t\omega\}-\frac{1}{2}\right)\right)$.
Using \eqref{eq:sn0} we get:

\begin{eqnarray}
\frac{s_{n0}}{s_{nt}} & = & \frac{(1+O(\omega^{n/5}))}{u_{t}}
\end{eqnarray}

Hence we can write:
\begin{eqnarray}
\prod_{t=1}^{q}\left(1-\frac{s_{n0}^{2}}{s_{nt}^{2}}\right) & = & \prod_{t=1}^{q}\left(1-\frac{1}{u_{t}^{2}}-\frac{O(\omega^{n/5})}{u_{t}^{2}}\right)\nonumber \\
 & = & \prod_{t=1}^{q}\left(1-\frac{1}{u_{t}^{2}}\right)\;\prod_{t=1}^{q}\left(1-\frac{O(\omega^{n/5})}{u_{t}^{2}-1}\right)
\end{eqnarray}

Now $\sum1/(u_{t}^{2}-1)$ converges (by comparison with $\sum1/t^{2}=\pi^{2}/6$)
and so $\sum\frac{O(\omega^{n/5})}{u_{t}^{2}-1}=O(\omega^{n/5})$,
so by Lemma \ref{lem:ProdBounds}: 
\begin{equation}
\prod_{t=1}^{q}\left(1-\frac{O(\omega^{n/5}}{u_{t}^{2}-1}\right)=1+O(\omega^{n/5})
\end{equation}
Similarly $\sum1/u_{t}^{2}$ also converges, but for this series we
need more information about the limit which we obtain as follows:
 
\begin{eqnarray}
\sum_{t=1}^{\infty}\frac{1}{u_{t}^{2}} & < & \frac{1}{u_{1}^{2}}+\sum_{t=2}^{\infty}\frac{1}{20(t-1)^{2}}\nonumber \\
 & < & 0.056+\pi^{2}/120\nonumber \\
 & < & 0.138
\end{eqnarray}
We now put $U_{q}=\prod_{t=1}^{q}\left(1-\frac{1}{u_{t}^{2}}\right)>1-\sum_{t=1}^{q}1/u_{t}^{2}>0.862$.
Note that $U_{q}$ is a descending sequence and bounded below, and
so converges to some constant $U_{\infty}>0.862$. (In fact we compute
$U_{\infty}\simeq0.928$). And 
\begin{eqnarray*}
1>\frac{U_{\infty}}{U_{q}} & = & \prod_{t=q+1}^{\infty}\left(1-\frac{1}{u_{t}^{2}}\right)>1-\frac{1}{20}\sum_{t=q+1}^{\infty}\frac{1}{(t-1)^{2}}\\
 & = & 1-O(q^{-1})
\end{eqnarray*}

Finally:
\begin{eqnarray*}
C_{n} & = & \prod_{t=1}^{F_{n/2}}\left(1-\frac{s_{n0}^{2}}{s_{nt}^{2}}\right)=\prod_{t=1}^{q}\left(1-\frac{s_{n0}^{2}}{s_{nt}^{2}}\right)\prod_{t=q+1}^{F_{n}/2}\left(1-\frac{s_{n0}^{2}}{s_{nt}^{2}}\right)\\
 & = & U_{\infty}\left(1+O(q^{-1})\right)\left(1+O(\omega^{n/5})\right)\left(1-O(\omega^{n/5})\right)\\
 & = & U_{\infty}\left(1+O(\omega^{n/5})\right)
\end{eqnarray*}

Hence 
\begin{equation}
\lim_{n\rightarrow\infty}C_{n}=U_{\infty}=\prod_{t=1}^{\infty}\left(1-\frac{1}{20\left(t-\frac{1}{\sqrt{5}}\left(\{t\omega\}-\frac{1}{2}\right)\right)^{2}}\right)^{2}\simeq0.928
\end{equation}

\section{\label{sec:EndMain}The Convergence of  $B_{n}=\prod_{t=1}^{F_{n}-1}s_{nt}/\left(2\sin\pi t/F_{n}\right)$}

In this step we show that $B_{n}$ converges to a strictly positive
limit. In the last section we saw that the proof of convergence was
complicated by the presence of non-alternating oscillations in sign.
We were able to circumvent this problem by relating the product to
one which converged absolutely, and involved a product of square terms
$\prod(1-1/r^{2})$. In this section we are unable to do this as the
absolute product behaves like $\prod(1+1/r)$ and diverges. The convergence
is therefore conditional and we are forced to estimate the compound
effects of the signed differences. 
\begin{thm}
\label{thm:Bn}The sequence $\log B_{n}$ converges to a finite limit,
and the sequence $B_{n}$ to a strictly positive limit
\end{thm}
We start by examining each term for $1\le t\le F_{n}-1$: 
\begin{eqnarray}
\frac{s_{nt}}{2\sin\pi t/F_{n}} & = & \frac{2\sin\pi(t/F_{n}-\omega^{n}\xi_{nt})}{2\sin\pi t/F_{n}}\nonumber \\
 & = & \cos\pi\omega^{n}\xi_{nt}-\cot\frac{\pi t}{F_{n}}\sin\pi\omega^{n}\xi_{nt}\nonumber \\
 & = & 1-2\sin^{2}\frac{\pi}{2}\omega^{n}\xi_{nt}-\cot\frac{\pi t}{F_{n}}\sin\pi\omega^{n}\xi_{nt}
\end{eqnarray}
Put $\alpha_{nt}=2\sin^{2}\frac{1}{2}\pi\left(\omega^{n}\xi_{nt}\right)$
and $h_{nt}=\cot\pi\left(\frac{t}{F_{n}}\right)\sin\pi\left(\omega^{n}\xi_{nt}\right)$
so that $B_{n}=\prod_{t=1}^{F_{n}-1}(1-\alpha_{nt}-h_{nt})$.

We first need an estimate for $h_{nt}$. Using $\cot x<1/x$ in $(0,\pi/2)$
and $\left|\xi_{nt}\right|<1/2$ gives us:

\begin{eqnarray}
\left|h_{nt}\right| & = & \cot\left(\frac{\pi t}{F_{n}}\right)\left|\sin\pi\left(\omega^{n}\xi_{nt}\right)\right|\nonumber \\
 & < & \frac{F_{n}}{\pi t}.\pi\omega^{n}\left|\xi_{nt}\right|\nonumber \\
 & < & \frac{1}{2\sqrt{5}t}\left(1-(-1)^{n}\omega^{2n}\right)\nonumber \\
 & < & \frac{1}{4t}\textrm{ for any }1\le t\le F_{n}/2\label{eq:hntEst1}
\end{eqnarray}

Also since $\left|\xi_{nt}\right|<1/2$ we have $0<\alpha_{nt}<\pi^{2}\omega^{2n}/8$.
Consequently $\log(1-\alpha_{nt}-h_{nt})=\log(1-h_{nt})+O(\omega^{2n})$
and we can sum over $t$ to obtain: 

\begin{eqnarray}
\left|\log B_{n}-\sum_{t=1}^{F_{n}-1}\log(1-h_{nt})\right| & = & O(\omega^{n})
\end{eqnarray}

Writing $B_{n}^{*}=\prod_{t=1}^{F_{n}-1}(1-h_{nt})$, this gives us:
\begin{equation}
B_{n}\sim B_{n}^{*}\label{eq:Bn}
\end{equation}

We proceed to investigate the product $B_{n}^{*}$. We start by observing:
\begin{equation}
\log B_{n}^{*}=\sum_{t=1}^{F_{n}-1}\log(1-h_{nt})=-\sum_{t=1}^{F_{n}-1}\sum_{k=1}^{\infty}\frac{1}{k}h_{nt}^{k}
\end{equation}

Using (from Lemma \ref{lem:baseSeq}) the symmetry $h_{nt}=h_{n(F_{n}-t)}$,
we obtain:
\begin{eqnarray}
\log B_{n}^{*} & = & -2\sum_{t=1}^{F_{n}/2}\sum_{k=1}^{\infty}\frac{1}{k}h_{nt}^{k}=-2\left(\sum_{t=1}^{F_{n}/2}h_{nt}+\sum_{t=1}^{F_{n}/2}\sum_{k=2}^{\infty}\frac{1}{k}h_{nt}^{k}\right)\label{eq:Bn*}
\end{eqnarray}

\subsection{Convergence of $\sum_{t=1}^{F_{n}/2}\sum_{k=2}^{\infty}\frac{1}{k}h_{nt}^{k}$}

Examining the right hand sum we find we can immediately take limits,
using \eqref{eq:hntEst1}: 
\begin{eqnarray*}
\lim_{n\rightarrow\infty}\sum_{t=1}^{F_{n}/2}\sum_{k=2}^{\infty}\frac{1}{k}\left|h_{nt}^{k}\right| & = & \sum_{t=1}^{\infty}\sum_{k=2}^{\infty}\frac{1}{k}\left|h_{nt}^{k}\right|<\sum_{t=1}^{\infty}\frac{h_{nt}^{2}}{1-\left|h_{nt}\right|}<\frac{\pi^{2}}{72}
\end{eqnarray*}

Hence the sum above is absolutely convergent, and hence convergent
to a limit we denote $L_{2}^{B}$, ie:
\begin{equation}
\lim_{n\rightarrow\infty}\sum_{t=1}^{F_{n}/2}\sum_{k=2}^{\infty}\frac{1}{k}h_{nt}^{k}=L_{2}^{B}\label{eq:LB2}
\end{equation}

\subsection{Convergence of $\sum_{t=1}^{F_{n}/2}h_{nt}$}

We are left in \eqref{eq:Bn*} with estimating the first sum $\sum_{t=1}^{F_{n}/2}h_{nt}$.
Our estimate of $\left|h_{nt}\right|<1/4t$ is not good enough to
help us here as its sum is the divergent harmonic series.

Put $H_{n}=\sum_{t=1}^{F_{n}/2}h_{nt}=\sum_{t=1}^{F_{n}/2}\cot\left(\frac{\pi t}{F_{n}}\right)\sin\pi\left(\omega^{n}\xi_{nt}\right)$
and $H_{n}^{*}=\sum_{t=1}^{F_{n}/2}\cot\left(\frac{\pi t}{F_{n}}\right)\sin\pi\left(\omega^{n}\xi_{\infty t}\right)$
(where in $H_{n}^{*}$we have simply replaced $\xi_{nt}$ with $\xi_{\infty t}$).
Note that for $1\le t\le F_{n}-1$ we have $\xi_{nt}-\xi_{\infty t}=t(-\omega)^{n}/F_{n}$
so that
\begin{equation}
H_{n}-H_{n}^{*}=\sum_{t=1}^{F_{n}/2}\cot\left(\frac{\pi t}{F_{n}}\right)\pi\omega^{n}\frac{t(-\omega)^{n}}{F_{n}}\left(1+O(\omega^{2n})\right)
\end{equation}

Note that for $x\in(0,\pi/2]$ we have $\cot x<x^{-1}$ and so
\begin{eqnarray}
\left|H_{n}-H_{n}^{*}\right| & < & \sum_{t=1}^{F_{n}/2}\left(\frac{\pi t}{F_{n}}\right)^{-1}\pi\omega^{n}\frac{t(\omega)^{n}}{F_{n}}\left(1+O(\omega^{2n})\right)=\omega^{2n}\sum_{t=1}^{F_{n}/2}\left(1+O(\omega^{2n})\right)\nonumber \\
 & = & \frac{\omega^{n}}{\sqrt{5}}+O(\omega^{3n})\label{eq:HnLimit}
\end{eqnarray}

so that $H_{n}-H_{n}^{*}\rightarrow0$. We now focus on $H_{n}^{*}$.
For the next step we will need to revert to summation using integer
limits. To do this note that if $F_{n}$ is even then $h_{F_{n}/2}=\cot\pi/2\,\sin\pi\left(\omega^{n}\xi_{F_{n}/2}\right)=0$
so we can ignore this term. So now we can put $M_{n}=\left\lfloor (F_{n}-1)/2\right\rfloor $
and use summation by parts to obtain:
\begin{eqnarray}
H_{n}^{*} & = & \sum_{t=1}^{M_{n}}\cot\left(\frac{\pi t}{F_{n}}\right)\sin\pi\left(\omega^{n}\xi_{\infty t}\right)\nonumber \\
 & = & \sum_{t=1}^{M_{n}-1}\left(\cot\left(\frac{\pi t}{F_{n}}\right)-\cot\left(\frac{\pi(t+1)}{F_{n}}\right)\right)\sum_{s=1}^{t}\sin\pi\left(\omega^{n}\xi_{\omega s}\right)\;+\cot\left(\frac{\pi M_{n}}{F_{n}}\right)\sin\pi\left(\omega^{n}\xi_{\omega M_{n}}\right)\label{eq:Hn*}
\end{eqnarray}

Recalling $\left|\xi_{\infty t}\right|<1/2$, the trailing term is
easily estimated as:

\begin{equation}
\left|\cot\left(\frac{\pi M_{n}}{F_{n}}\right)\sin\pi\left(\omega^{n}\xi_{\omega M_{n}}\right)\right|<\left(\frac{\pi}{2F_{n}}\right)\left(\frac{\pi\omega^{n}}{2}\right)+O(\omega^{4n})=O(\omega^{2n})\label{eq:T1}
\end{equation}

We can now take limits on \eqref{eq:Hn*} to obtain, writing $\alpha=\pi/F_{n}$,
$C_{nt}=\cot t\alpha-\cot(t+1)\alpha$ and $S_{nt}=\sum_{s=1}^{t}\sin\pi\left(\omega^{n}\xi_{\omega s}\right)$:
\begin{equation}
\lim_{n\rightarrow\infty}H_{n}^{*}=\lim_{n\rightarrow\infty}\sum_{t=1}^{M_{n}-1}C_{nt}S_{nt}\label{eq:Hn*Lim}
\end{equation}

\subsubsection{The order of the cotangent difference}

We estimate the cotangent difference as follows:
\begin{eqnarray}
0<C_{nt}=\cot t\alpha-\cot(t+1)\alpha & = & \frac{\sin(t+1)\alpha\cos t\alpha-\cos(t+1)\alpha\sin t\alpha}{\sin t\alpha\sin(t+1)\alpha}\nonumber \\
 & = & \frac{2\sin\alpha}{\cos\alpha-\cos(2t+1)\alpha}\nonumber \\
 & = & \frac{\sin\alpha}{\sin^{2}(t+\frac{1}{2})\alpha-\sin^{2}\alpha}
\end{eqnarray}

Expanding $\alpha$, and noting from Lemma \ref{lem:ConvexSine} $(\pi/2)\sin x>x$
for $x\in(0,\pi/2)$ we get:
\begin{eqnarray}
0<C_{nt} & < & \frac{\pi F_{n}(1+O(F_{n}^{-2})}{(2t+1)^{2}-1}\nonumber \\
 & < & \frac{\pi F_{n}(1+O(F_{n}^{-2})}{2t^{2}}\label{eq:CotDiffEst}
\end{eqnarray}

\subsubsection{\label{sub:SumSt}The order of the partial sums $S_{nt}=\sum_{s=1}^{t}\sin\pi\left(\omega^{n}\xi_{\omega s}\right)$ }

\begin{figure}
\noindent \centering{}\includegraphics[scale=0.5]{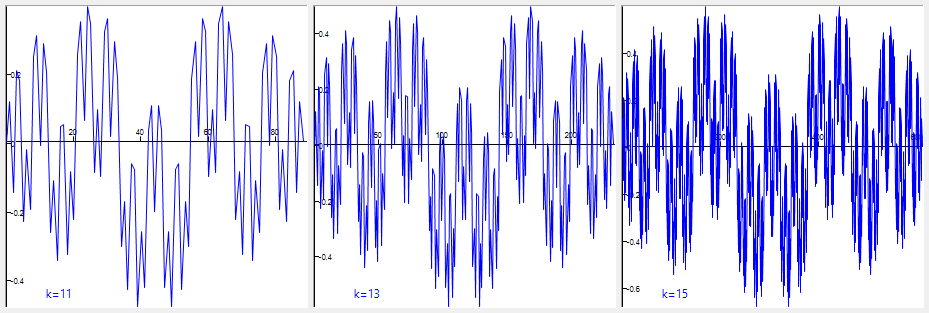}\protect\caption{Renormalised graph of $\left(S_{kt}/\omega^{n}\right)$ against $t$
for $t=0...F_{k}$. Note the slowly growing norm.}
\end{figure}

In this step we establish an estimate for $S_{nt}$ in terms of $t$
and $n$. However we will also need introduce a generalised $S_{nt}(\theta)$
in order to accommodate a dependency on a starting phase angle $\theta$.
Our basic approach will be to find an estimate for $S_{nF_{k}}(\theta)$
and then express $S_{nt}(0)=S_{nt}$ as a sum of terms involving $S_{nF_{k}}(\theta)$.

Recall from \ref{def:FibSumLength} that we can represent $t\ge1$
as a Fibonacci sum $t=\sum_{s=1}^{m}b_{s}F_{s}$ where $m(t)$ is
the largest integer such that $F_{m}\le t$. Define $t_{m}=0$, and
for $0\le s\le m-1$ define $t_{s}=t_{s+1}+b_{s+1}F_{s+1}=\sum_{u=s+1}^{m}b_{u}F_{u}$
so that $t_{0}=t$.

For $1\le r\le F_{n}-1$ we now introduce a generalised $\xi_{\omega r}(\theta)=\{\theta+r\omega\}-1/2$
so that our $\xi_{\infty t}$ of the previous section is now represented
by $\xi_{\infty t}(0)$. We can now use the Fibonacci representation
of $t$ to split the sum $S_{nt}$ into segments of length $b_{s}F_{s}$
: 

\begin{equation}
S_{nt}=\sum_{r=1}^{t}\sin\pi\omega^{n}\xi_{\omega s}(0)=\sum_{s=1}^{m}\sum_{r=1}^{b_{s}F_{s}}\sin\pi\omega^{n}\xi_{\omega r}\left(t_{s}\omega\right)
\end{equation}
.

We now introduce a generalised $S_{nt}(\theta)=\sum_{r=1}^{t}\sin\pi\left(\omega^{n}\xi_{\omega r}(\theta)\right)$
which allows us to write for $1\le t\le F_{n}-1$ 
\begin{equation}
S_{nt}=S_{nt}(0)=\sum_{s=1}^{m(t)}b_{s}S_{nF_{s}}(t_{s}\omega)\label{eq:Sr}
\end{equation}
We proceed to study the order of the terms $S_{nF_{s}}(\theta)$.

\begin{lem}
Let $p/q$ be a convergent of any real $\alpha$. Then for any real
$\theta$ 
\[
\left|\sum_{i=1}^{q}\left(\{\theta+i\alpha\}-\frac{1}{2}\right)\right|<\frac{3}{2}
\]
\end{lem}
\begin{proof}
We can assume without loss of generality that $\alpha,\theta\in[0,1)$.
Since $p/q$ is a convergent of $\alpha$ we have $\alpha-p/q=\nu/q^{2}$
for some $\left|\nu\right|<1$.

Now $\theta=k/q+\phi$ for some $0\le k<q$ and $0\le\phi<1/q$, and
so $\theta+i\alpha=(k+ip)/q+\phi+i\nu/q^{2}$.

Suppose $\nu\ge0$, then for $1\le i\le q$ we have $ $$(k+ip)/q\le\theta+i\alpha<(k+ip+2)/q$
and so $ $$\left\{ (k+ip)/q\right\} \le\left\{ \theta+i\alpha\right\} $
with the one exception that when $k+ip\equiv-1\bmod q$ we may have
$\phi+i\nu/q^{2}\ge1/q$ and then we can only write $\left\{ (k+ip)/q\right\} -(q-1)/q\le\left\{ \theta+i\alpha\right\} $.

Now $(p,q)=1$, and so as $i$ runs through $1,\ldots,q$, $k+ip$
runs through a complete set of residues $0,\ldots,q-1$ $\bmod q$
, and hence
\begin{equation}
\sum_{i=1}^{q}\{\theta+i\alpha\}\ge\left(\sum_{j=0}^{q-1}\frac{j}{q}\right)-\frac{q-1}{q}=\frac{1}{2}(q-1)-\frac{q-1}{q}\label{eq:sumLowerBound}
\end{equation}

Similarly for $\nu<0$ we have $(k+ip-1)/q<\theta+i\alpha<(k+ip+1)/q$
and so $ $$\left\{ (k+ip-1)/q\right\} <\left\{ \theta+i\alpha\right\} $
with the one exception that when $k+ip-1\equiv-1\bmod q$ we may have
$\phi+i\nu/q^{2}\ge0$ and then we can only write $\left\{ (k+ip-1)/q\right\} -(q-1)/q\le\left\{ \theta+i\alpha\right\} $.
Now summing as before also gives \eqref{eq:sumLowerBound}, and so
this holds for any $\left|\nu\right|<1$. We can immediately deduce\foreignlanguage{english}{
\begin{equation}
\sum_{i=1}^{q}\left(\{\theta+i\alpha\}-\frac{1}{2}\right)>-\frac{3}{2}
\end{equation}
}

We now examine the upper bound of the sum. For $\nu\ge0$ we have
\begin{equation}
\sum_{i=1}^{q}\{\theta+i\alpha\}\le\sum_{j=0}^{q-1}\frac{j}{q}+\sum_{i=1}^{q}\left(\phi+\frac{i\nu}{q^{2}}\right)=\frac{1}{2}(q-1)+q\phi+\frac{1}{2}(q+1)\frac{\nu}{q}<\frac{1}{2}(q-1)+1+\frac{1}{2}(1+\frac{1}{q})=\frac{1}{2}q+1+\frac{1}{2q}
\end{equation}

whilst for $\nu<0$ we get 
\begin{equation}
\sum_{i=1}^{q}\{\theta+i\alpha\}\le\sum_{j=1}^{q}\frac{j}{q}+\sum_{i=1}^{q}\left(\frac{i\nu}{q^{2}}\right)=\frac{1}{2}(q+1)+\frac{1}{2}(q+1)\frac{\nu}{q}<\frac{1}{2}(q+1)
\end{equation}

By adding $\sum_{i=1}^{q}(-1/2)=-(1/2)q$ to the two upper bound inequalities,
the result follows by combining the three bounds obtained
\end{proof}
We now fix $n$, and use the lemma to estimate $S_{nF_{i}}$ for $1\le i<n$.
We can do this by noting that $F_{i-1}/F_{i}$ is a convergent to
$\omega$. 

Using $\sin x=x+O(x^{3})$ and from \eqref{eq:FnEst} $F_{i}\omega^{2n}\le F_{n}\omega^{2n}=O(\omega^{n})$,
we can now apply the lemma to estimate $S_{nF_{i}}(\theta)$ as follows:
\begin{eqnarray}
|S_{nF_{i}}(\theta)| & = & \left|\sum_{p=1}^{F_{i}}\sin\pi\omega^{n}\left(\{\theta+p\omega\}-\frac{1}{2}\right)\right|\nonumber \\
 & = & \pi\omega^{n}\left|\sum_{p=1}^{F_{i}}\left(\{\theta+p\omega\}-\frac{1}{2}+O(\omega^{2n})\right)\right|\nonumber \\
 & < & \pi\omega^{n}\left(\frac{3}{2}+O(\omega^{n})\right)
\end{eqnarray}

Note that the $O(\omega^{n})$ term has no dependency on $i$.

We are now in a position to estimate $S_{nt}$ for $1\le t\le F_{n}-1$,
using \eqref{eq:Sr} for $1\le t\le F_{n}-1$:

\begin{eqnarray*}
\left|S_{nt}(0)\right| & = & \left|\sum_{s=1}^{m(t)}b_{s}S_{nF_{s}}(t_{s}\omega)\right|<\sum_{s=1}^{m}b_{s}\pi\omega^{n}\left(\frac{3}{2}+O(\omega^{n})\right)\\
 & < & \pi\omega^{n}\left(\frac{3}{2}+O(\omega^{n})\right)\sum_{s=1}^{m}b_{s}
\end{eqnarray*}

Now $\sum_{s=1}^{m}b_{s}$ is the Fibonacci length of $t$ and by
Lemma \ref{lem:FLest} $\sum_{s=1}^{m}b_{s}\le\left\lfloor (\log t\,+1)/\log(2+\omega)\right\rfloor $
and so we have established:
\begin{lem}
\label{thm:OrderPartialSums} For $1\le t\le F_{n}-1$, the partial
sums $S_{nt}=\sum_{s=1}^{t}\sin\pi\left(\omega^{n}\xi_{\infty t}\right)$
satisfy

\begin{equation}
S_{nt}<\frac{3}{2}\pi\omega^{n}\left\lfloor (\log t\,+1)/\log(2+\omega)\right\rfloor +O(n\omega^{2n})
\end{equation}

In particular we can find a $K$ independent of $n$ such that $\left|S_{nt}\right|<K\omega^{n}(\log t\,+1)$
\end{lem}

\subsubsection{Conclusion of proof of convergence of the first sum}

From Theorem \ref{thm:OrderPartialSums} in section \ref{sub:SumSt}
we have $\left|S_{nt}\right|=\left|\sum_{s=1}^{t}\sin\pi\left(\omega^{n}\xi_{\omega s}\right)\right|<K\omega^{n}(\log t\,+1)$
for some $K$ independent of $n$. Combining this with \eqref{eq:CotDiffEst}
we get $\left|C_{nt}S_{nt}\right|<\pi K(1+O(\omega^{2n}))(\log t\,+1)/t^{2}$.
But $\sum(\log t\,+1)/t^{2}$ is absolutely convergent, so putting
$K_{2}=\pi K\sum_{1}^{\infty}(\log t\,+1)/t^{2}$ in \eqref{eq:Hn*Lim}
we get: 
\begin{equation}
\lim_{n\rightarrow\infty}\sum_{t=1}^{M_{n}-1}\left|\left(\cot\left(\frac{\pi t}{F_{n}}\right)-\cot\left(\frac{\pi(t+1)}{F_{n}}\right)\right)\sum_{s=1}^{t}\sin\pi\left(\omega^{n}\xi_{\omega s}\right)\right|\le K_{2}
\end{equation}

So the sum above is absolutely convergent, and hence converges to
a limit $L_{1}^{B}$. From \eqref{eq:HnLimit} and \eqref{eq:Hn*Lim}
this gives us: 
\begin{equation}
H_{n}\rightarrow H_{n}^{*}\longrightarrow L_{1}^{B}\label{eq:LB1}
\end{equation}

\subsection{Conclusion of proof of convergence of $B_{n}$}

Combining \eqref{eq:Bn}, \eqref{eq:Bn*}, \eqref{eq:LB2} and \eqref{eq:LB1}
and gives us finally
\begin{equation}
\log B_{n}\rightarrow-2\left(L_{1}^{B}+L_{2}^{B}\right)
\end{equation}

and noting that both limits are finite establishes Theorem \ref{thm:Bn}.

\section{\label{sec:Consequences}Two additional results}

In this section we show how the other two results of Theorem \ref{prop:K=000026T}
flow from our main result $P_{F_{n}}(\omega)\rightarrow c$.

The first result is really just a direct corollary of our main result.

\subsection{The convergence of $P_{F_{n}-1}(\omega)/F_{n}$}
\begin{cor}
\label{prop:PFn-1}The sequence $P_{F_{n}-1}(\omega)/F_{n}$ converges
to $c\sqrt{5}/2\pi$ where $c$ is the limit of the sequence $P_{F_{n}}(\omega)$ \end{cor}
\begin{proof}
Since $P_{F_{n}-1}(\omega)=P_{F_{n}}(\omega)/2\sin\pi\omega^{n}\sim c/2\pi\omega^{n}$,
the result follows from $F_{n}\sim\omega^{-n}/\sqrt{5}$ 
\end{proof}

\subsection{The power law growth of $P_{k}(\omega)$ for general $k$}

We now turn to the more important result that the growth and decay
of $P_{k}(\omega)$ is bounded by power laws, specifically:
\begin{thm}
There are real constants $K_{1}\le0<1\le K_{2}$ independent of $k$
such that for $k\ge1$ we have $k^{K_{1}}\le P_{k}(\omega)\le k^{K_{2}}$ 
\end{thm}
The main part of the proof is to establish that these constants exist.
If they do then our main result ($P_{F_{n}}(\omega)\longrightarrow c$)
shows we must have $K_{1}\le0$, and Proposition \ref{prop:PFn-1}
($P_{F_{n}-1}(\omega)/F_{n}\longrightarrow c\sqrt{5}/2\pi$) shows
we must have $K_{2}\ge1$.

Knill and Tangerman provide an outline proof of existence in the logarithmic
case, but appear to make an assumption which, although correct, seems
to us to require its own proof. We will give the outline proof here,
and then complete it rigorously. 

Recall from section \ref{sub:Fibonacci-representation} that we can
express any integer $k\ge1$ as a sum of Fibonacci numbers $\sum_{s=1}^{m}b_{s}F_{s}$
subject to the rules ($b_{s}\in\{0,1\},b_{m}=1,\, b_{r+1}=1\Rightarrow b_{r}=0$).
For $0\le s\le m-1$ put $k_{s}=\sum_{u=s+1}^{m}b_{u}F_{u},\, k_{m}=0$
so that for $m>1$ we can split the overall product into sub-products
of length $b_{s}F_{s}$ (regarding the empty product as $1$) to get:
\begin{eqnarray}
P_{k}(\omega) & = & \prod_{r=1}^{k}|2\sin\pi(r\omega)|\nonumber \\
 & = & \prod_{r=1}^{b_{m}F_{m}}|2\sin\pi(r\omega)|\prod_{r=1}^{b_{m-1}F_{m-1}}|2\sin\pi(r\omega+b_{m}F_{m}\omega)|\prod_{r=1}^{b_{m-2}F_{m-2}}|2\sin\pi(r\omega+(b_{m}F_{m}+b_{m-1}F_{m-1})\omega)|\,...\nonumber \\
 & = & \prod_{s=1}^{m}\prod_{r=1}^{b_{s}F_{s}}|2\sin\pi(r\omega+k_{s}\omega)|\label{eq:pkw}
\end{eqnarray}

Now the term for $s=m$ of this product is $\prod_{r=1}^{b_{m}F_{m}}|2\sin\pi(r\omega)|=P_{F_{m}}(\omega)\sim c$
(by the main result of this paper), and it is also strictly positive,
so that we can find constants $0<C_{1}<C_{2}$ bounding $P_{F_{m}}(\omega)$
for all $m$. 
\begin{conjecture}
\label{conj:constantsExist}Assume that we can choose real constants
$C_{1},C_{2}$ with $0<C_{1}<C_{2}$ such that they bound all the
terms in \eqref{eq:pkw}, ie so that 
\begin{equation}
C_{1}\le\prod_{r=1}^{b_{s}F_{s}}|2\sin\pi(r\omega+k_{s}\omega)|\le C_{2}\label{eq:conjecture}
\end{equation}
 for each $1\le s\le m$. (Note that in order to bound empty products
this requires $C_{1}\le1\le C_{2}$). 
\end{conjecture}
Then we have from \eqref{eq:pkw}: 
\begin{equation}
C_{1}^{m}\le P_{k}(\omega)\le C_{2}^{m}\label{eq:PkPowerBound1}
\end{equation}
Now for $m\ge1$, using \eqref{eq:Fnwn} we obtain $\log k\ge\log F_{m}=\log(\omega^{-m}/\sqrt{5})\left(1-(-1)^{m}\omega^{2m}\right)>m\log\omega^{-1}-\log\sqrt{5}-1$
which gives $m<(\log c'k)/\log\omega^{-1}$ for some constant $c'$.
So for real but not necessarily positive constants $K_{1},K_{2}$:
\begin{equation}
k^{K_{1}}<P_{k}(\omega)<k^{K_{2}}\label{eq:BoundedPk}
\end{equation}

This is essentially an amplified version of the outline proof provided
by Knill and Tangerman, although we have provided it in multiplicative
form, rather than the additive (logarithmic) form used in the aforementioned
paper. The assumption in conjecture \ref{conj:constantsExist} is
in fact correct (and it is trivial if $b_{s}=0$), but a proof does
not appear trivial for $b_{s}=1$, and so we provide  one here. 

Since the case $b_{s}=0$ is trivial we need deal only with $b_{s}=1$.
By the rules of the Fibonacci decomposition (see Lemma \ref{lem:FLest}),
$b_{r}=1$ implies $b_{r+1}=0$. Hence $k_{s}\omega=\sum_{u=s+2}^{m}b_{u}F_{u}\omega=N+\sum_{u=s+2}^{m}-b_{u}(-\omega)^{u}$
for some integer $N$. Now $\left|\sum_{u=s+2}^{m}-b_{u}(-\omega)^{u}\right|\le\omega^{s+2}(1+\omega^{2}+\omega^{4}...)<\omega^{s+1}$.
Hence the conjecture is proved if we can prove the slightly more general
assertion:
\begin{lem}
\label{lem:BoundedIC}There are real constants $C_{1},C_{2}$ satisfying
$0<C_{1}\le1\le C_{2}$ such that $C_{1}\le\prod_{1}^{F_{n}}|2\sin\pi(r\omega+\alpha)|\le C_{2}$
whenever $n\ge2$ and $\left|\alpha\right|\le\omega^{n+1}$ . 
\end{lem}
Note the lemma does not hold for $n=1$ as $\prod_{1}^{F_{n}}|2\sin\pi(r\omega+\alpha)|=0$
for $\alpha=\omega^{2}$.

We begin by expanding the sine product as follows:

\begin{eqnarray}
\prod_{1}^{F_{n}}|2\sin\pi(r\omega+\alpha)| & = & \prod_{1}^{F_{n}}|2\sin\pi(r\omega)|\left|\cos\pi\alpha+\cot\pi r\omega.\sin\pi\alpha\right|\label{eq:PkOfAlphaMult}
\end{eqnarray}

For $n\ge2$ and $\left|\alpha\right|\le\omega^{n+1}$ it easy to
calculate that $\cos\pi\alpha+\cot\pi r\omega.\sin\pi\alpha>0$. We
can therefore take logs of the product above to obtain:

\begin{equation}
\log\prod_{1}^{F_{n}}|2\sin\pi(r\omega+\alpha)|=\log P_{F_{n}}(\omega)+\sum_{1}^{F_{n}}\log\left(1-2\sin^{2}\frac{\pi\alpha}{2}+\cot\pi r\omega.\sin\pi\alpha\right)\label{eq:PkOfAlpha}
\end{equation}

Since $P_{F_{n}}(\omega)$ is already suitably bounded by the main
result of this paper, it remains to show that the log sum is bounded
above and below.

We begin with establishing the upper bound as this is slightly more
straightforward than the lower bound.

\subsubsection{The upper bound on the growth rate}

We use $\log(1+x)\le x$ for $x\in(-1,1]$), and $\sin x>2x/\pi$
from Lemma \ref{lem:ConvexSine} to obtain:

\begin{eqnarray}
\sum_{r=1}^{F_{n}}\log\left(1-2\sin^{2}\frac{\pi\alpha}{2}+\cot\pi r\omega.\sin\pi\alpha\right) & < & \sum_{r=1}^{F_{n}}\left(-2\sin^{2}\frac{\pi\alpha}{2}+\cot\pi r\omega.\sin\pi\alpha\right)\nonumber \\
 & = & -2F_{n}\sin^{2}\frac{\pi\alpha}{2}+\sin\pi\alpha\sum_{r=1}^{F_{n}}\cot\pi r\omega\nonumber \\
 & < & \left(-2\frac{\omega^{-n}}{\sqrt{5}}\left(1+\omega^{2n}\right)\omega^{2n+2}\right)+\left(\pi\omega^{n+1}\left|\sum_{r=1}^{F_{n}}\cot\pi r\omega\right|\right)\label{eq:cotsin}
\end{eqnarray}

We now examine the sum $S_{F_{n}}(\omega)=\sum_{r=1}^{F_{n}}\cot\pi r\omega$.
The sum $S_{k}(\omega)=\sum_{r=1}^{k}\cot\pi r\omega$ clearly has
interesting relationships with our original product $P_{k}(\omega)=\prod_{r=1}^{k}|2\sin\pi r\omega|$,
and indeed it shows definite self-similar characteristics over Fibonacci
intervals (see Figure \ref{fig:CotSum}). We conjecture that the renormalised
functions are converging, but for our current purposes we need only
to establish bounds, which we proceed to do as follows.

\begin{figure}
\noindent \begin{centering}
\includegraphics[scale=0.5]{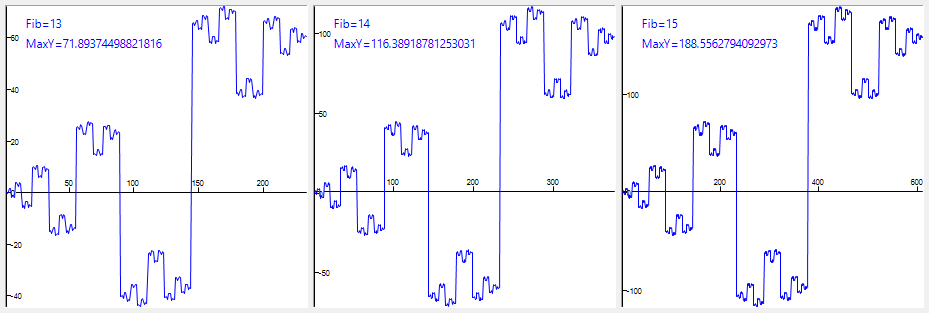}
\par\end{centering}

\protect\caption{\label{fig:CotSum}The renormalised graphs of $(-1)^{n}\sum_{1}^{k}\cot\pi r\omega$
over the Fibonacci interval $[1,F_{n}-1]$ for $n=13,14,15$. Note
the remarkable scaling similarity.}
\end{figure}

For $n\ge3$, let $s(r)=[rF_{n-1}]$ so that as $r$ runs through
the values $1,...,F_{n}-1$ so does $s$.

We consider $n$ odd so that $\omega=\left(F_{n-1}+\omega^{n}\right)/F_{n}$.
It follows that the fractional part of $r\omega$ lies in the interval
$(s/F_{n},\,(s+1)/F_{n})$. Using $\cot x>\cot\left(x+\theta\right)$
when $x,x+\theta\in(0,\pi)$ and $\theta>0$ gives us $\cot\pi s/F_{n}>\cot\pi r\omega$.
We now use the fact that $\cot\pi x+\cot\pi(1-x)=0$ to obtain:
\begin{equation}
\sum_{r=1}^{F_{n}-1}\cot\pi r\omega<\sum_{r=1}^{F_{n}-1}\cot\pi\frac{s}{F_{n}}=0\label{eq:oddCotUpper}
\end{equation}

Similarly we have $\cot\pi r\omega>\cot\pi(s+1)/F_{n})$ but now $s+1$
runs through the values $2,..F_{n}$. However the value of $s+1=F_{n}$
results in a singularity of $\cot\pi$ so we treat separately the
case $s=F_{n}-1$. Let $r^{*}$ be the value of $r$ which satisfies
$s(r^{*})=F_{n}-1$. Then from Lemma \ref{lem:Fn-1Inverse} we have
$[F_{n-1}]^{-1}=[-F_{n-1}]$ for $n$ odd, and so $[r^{*}]=[F_{n-1}]^{-1}[F_{n}-1]=[F_{n-1}]$,
giving $r^{*}=F_{n-1}$. This gives us:
\begin{eqnarray}
\sum_{r=1}^{F_{n}-1}\cot\pi r\omega & > & \sum_{s=1}^{F_{n}-2}\cot\pi\frac{s+1}{F_{n}}+\cot\pi r^{*}\omega\nonumber \\
 & = & \left(\sum_{s=1}^{F_{n}-1}\cot\pi\frac{s}{F_{n}}-\cot\frac{\pi}{F_{n}}\right)+\cot\left(\pi F_{n-1}\omega\right)\nonumber \\
 & = & \left(0-\cot\frac{\pi}{F_{n}}\right)-\cot\pi\omega^{n-1}\nonumber \\
 & > & -\frac{1}{\pi}\left(F_{n}+\omega^{-(n-1)}\right)\,\textrm{for odd }n\ge3\label{eq:oddCotLower}
\end{eqnarray}

But for $n$ odd, $0<\cot\pi F_{n}\omega=\cot\pi\omega^{n}<1/\pi\omega^{n}$
and so we can add this to inequalities \eqref{eq:oddCotUpper} and
\eqref{eq:oddCotLower} to obtain for odd $n\ge3$, that: 
\begin{equation}
-\frac{1}{\pi}\left(\frac{1}{\sqrt{5}}\left(1+\omega^{2n}\right)+\omega\right)<\omega^{n}\sum_{r=1}^{F_{n}}\cot\pi r\omega<\frac{1}{\pi}\label{eq:SumCotA}
\end{equation}

By reversing signs appropriately, the same argument establishes an
equivalent result for even $n\ge4$, which in fact is easily verified
to hold for $n=2$: 

\begin{equation}
\frac{1}{\pi}\left(\frac{1}{\sqrt{5}}\left(1+\omega^{2n}\right)+\omega\right)>\omega^{n}\sum_{r=1}^{F_{n}}\cot\pi r\omega>\frac{-1}{\pi}\label{eq:SumCotB}
\end{equation}

In both cases (even and odd) the left hand term is slightly larger
in absolute value than the right. From \eqref{eq:cotsin}, we therefore
obtain for $n\ge2$:
\begin{eqnarray*}
\sum_{r=1}^{F_{n}}\log\left(1-2\sin^{2}\frac{\pi\alpha}{2}+\cot\pi r\omega.\sin\pi\alpha\right) & < & \left(-2\frac{\omega^{n+2}}{\sqrt{5}}\left(1+\omega^{2n}\right)\right)+\omega\left(\frac{1}{\sqrt{5}}\left(1+\omega^{2n}\right)+\omega\right)\\
 & < & \omega\left(\frac{1}{\sqrt{5}}+\omega\right)
\end{eqnarray*}
This establishes the upper bound we needed, and also in \eqref{eq:PkOfAlphaMult}
we now have for $n\ge2$: 
\[
\prod_{1}^{F_{n}}|2\sin\pi(r\omega+\alpha)|<P_{F_{n}}(\omega)e^{\omega\left(\frac{1}{\sqrt{5}}+\omega\right)}
\]

This now also establishes the upper bound in \eqref{eq:PkPowerBound1}
and hence also in \eqref{eq:BoundedPk}. We now turn to the lower
bound.

\subsubsection{The lower bound on the growth rate}

For the upper bound we were able to use the standard result that $\log(1+x)>x$.
We now need a lower bound for the logarithm The following lemma provides
this:
\begin{lem}
For real $x>-0.683$ we have $\log(1+x)\ge x-x^{2}$ \end{lem}
\begin{proof}
For $x>-1$ put $f(x)=\log(1+x)-(x-x^{2})$. Note the function is
continuous on $(-1,\infty)$ and that $f(0)=0$. It is easy to verify
that this has critical points at $x=0,-0.5$ and the derivative is
positive on $(0,\infty)$ and negative on $(-0.5,0)$ so that the
function itself is positive on these two intervals. On $(-1,-0.5)$
the derivative is negative so the function descends with descending
$x$ from its maximum at $x=-0.5$ to a zero in $(-1,-0.5)$. A numerical
calculation shows the root lies just below $x=-0.683$.
\end{proof}
We wish to apply the lemma to the expression $\sum_{r=1}^{F_{n}}\log\left(1-2\sin^{2}\frac{\pi\alpha}{2}+\cot\pi r\omega.\sin\pi\alpha\right)$
from \eqref{eq:cotsin}. To do this we must first establish that $-2\sin^{2}\frac{\pi\alpha}{2}+\cot\pi r\omega.\sin\pi\alpha>-0.683$.
Now for $n\ge4$ we have:

\begin{eqnarray*}
|-2\sin^{2}\frac{\pi\alpha}{2}+\cot\pi r\omega.\sin\pi\alpha| & < & 2\frac{\pi^{2}\alpha^{2}}{4}+\pi\alpha\cot\pi\omega^{n}\\
 & < & \frac{\pi^{2}\omega^{2n+2}}{2}+\pi\omega^{n+1}\frac{1}{\pi\omega^{n}(1-\pi^{2}\omega^{2n}/6)}\\
 & < & \frac{\pi^{2}\omega^{10}}{2}+\frac{\omega}{(1-\pi^{2}\omega^{8}/6)}\\
 & < & 0.681
\end{eqnarray*}

We can now apply the lemma to obtain: 
\begin{eqnarray}
\sum_{r=1}^{F_{n}}\log\left(1-2\sin^{2}\frac{\pi\alpha}{2}+\cot\pi r\omega.\sin\pi\alpha\right) & \ge & \sum_{r=1}^{F_{n}}\left(-2\sin^{2}\frac{\pi\alpha}{2}+\cot\pi r\omega.\sin\pi\alpha\right)\nonumber \\
 &  & -\sum_{r=1}^{F_{n}}\left(-2\sin^{2}\frac{\pi\alpha}{2}+\cot\pi r\omega.\sin\pi\alpha\right)^{2}\nonumber \\
 & = & \sum_{r=1}^{F_{n}}-2\sin^{2}\frac{\pi\alpha}{2}-4\sin^{4}\frac{\pi\alpha}{2}\nonumber \\
 &  & +\sum_{r=1}^{F_{n}}\left(1+4\sin^{2}\frac{\pi\alpha}{2}\right)\cot\pi r\omega.\sin\pi\alpha-\left(\cot\pi r\omega.\sin\pi\alpha\right)^{2}\nonumber \\
 & \ge & F_{n}\left(-2\left(\frac{\pi\omega^{n+1}}{2}\right)^{2}-4\left(\frac{\pi\omega^{n+1}}{2}\right)^{4}\right)\label{eq:LowerBound}\\
 &  & -\left(1+4\left(\frac{\pi\omega^{n+1}}{2}\right)^{2}\right)\left|\sum_{r=1}^{F_{n}}\cot\pi r\omega.\sin\pi\alpha\right|\nonumber \\
 &  & -\sum_{r=1}^{F_{n}}\left(\cot\pi r\omega.\sin\pi\alpha\right)^{2}\nonumber 
\end{eqnarray}

The first term is clearly bounded below. From \eqref{eq:SumCotA},\eqref{eq:SumCotB}
and for $n\ge2$, we have $\left|\sum\cot\pi r\omega.\sin\pi\alpha\right|<\omega\left(\frac{1}{\sqrt{5}}\left(1+\omega^{2n}\right)+\omega\right)$,
and so the second term is also bounded below. It remains to show that
the third term is bounded below. Using Lemma \ref{lem:ConvexSine}
(and allowing for $\alpha=0$) we have for $n\ge1$: 
\begin{eqnarray*}
\sum_{r=1}^{F_{n}}\left(\cot\pi r\omega.\sin\pi\alpha\right)^{2} & \le & \left(\pi\alpha\right)^{2}\sum_{r=1}^{F_{n}}\cot^{2}\pi r\omega
\end{eqnarray*}

Using the same argument with $s(r)=[rF_{n-1}]$ as for the upper bound,
we obtain for $n\ge3$, $\cot^{2}\pi r\omega<\cot^{2}\pi s/F_{n}$
for $0\le s\le\left\lfloor \frac{1}{2}F_{n}-1\right\rfloor $ and
for $n\ge4$ it also gives $\cot^{2}\pi r\omega<\cot^{2}(\pi(s+1)/F_{n})$
for $\left\lceil \frac{1}{2}F_{n}\right\rceil \le s\le F_{n}-1$.
There is a special case: when $n\ge4$ and $F_{n}$ is odd there is
an uncovered interval $[\frac{1}{2}(F_{n}-1)/F_{n},\frac{1}{2}(F_{n}+1)/F_{n}]$,
but here again $\cot^{2}\pi r\omega<\cot^{2}\pi s/F_{n}$ for $s=\frac{1}{2}(F_{n}-1)=\left\lfloor \frac{1}{2}F_{n}\right\rfloor $.
We are now almost ready to sum over $r$, but we again need to take
care of singularities, and these occur this time at $s=0,F_{n}-1$,
corresponding to $r=F_{n},[(-1)^{n}F_{n-1}]$. Hence for $n\ge4$,
using $\left|\cot x\right|<\left|1/x\right|$
\begin{eqnarray}
\sum_{r=1}^{F_{n}}\cot^{2}\pi r\omega & < & \left(\sum_{s=1}^{\left\lfloor \frac{1}{2}F_{n}\right\rfloor }\cot^{2}\pi\frac{s}{F_{n}}+\sum_{s=\left\lceil \frac{1}{2}F_{n}\right\rceil }^{F_{n}-2}\cot^{2}\pi\frac{s+1}{F_{n}}\right)+\cot^{2}\pi F_{n}\omega+\cot^{2}\pi F_{n-1}\omega\nonumber \\
 & < & 2\sum_{s=1}^{\left\lfloor \frac{1}{2}F_{n}\right\rfloor }\left(\frac{F_{n}}{\pi s}\right)^{2}+\left(\frac{1}{\pi\omega^{n}}\right)^{2}+\left(\frac{1}{\pi\omega^{n-1}}\right)^{2}\nonumber \\
 & < & \frac{1}{\pi^{2}}F_{n}^{2}\left(\frac{\pi^{2}}{6}\right)+\frac{1+\omega^{2}}{\pi^{2}\omega^{2n}}\label{eq:cot2ub}
\end{eqnarray}

Hence for $n\ge4$:
\begin{eqnarray*}
\sum_{r=1}^{F_{n}}\left(\cot\pi r\omega.\sin\pi\alpha\right)^{2} & < & \pi^{2}\omega^{2n+2}\left(\frac{1}{6}\frac{\omega^{-2n}}{5}\left(1+\omega^{2n}\right)+\frac{1+\omega^{2}}{\pi^{2}\omega^{2n}}\right)\\
 & = & \frac{\omega^{2}}{30}\left(\pi^{2}(1+\omega^{8})+30(1+\omega^{2})\right)
\end{eqnarray*}
Hence the third term in \eqref{eq:LowerBound} is also bounded below,
and the lower bound we needed for this log sum is also established
for $n\ge4$. Hence for $n\ge4$, \eqref{eq:PkOfAlphaMult} is bounded
below by a strictly positive constant, and in fact it is easily verified
that this is also true for $n=2,3$, finally establishing Lemma \ref{lem:BoundedIC}.

This now also establishes the lower bound in \eqref{eq:PkPowerBound1}
and hence also in \eqref{eq:BoundedPk}.

\section{Conclusion }

In this paper we studied Sudler's sine product in the important special
case where $\omega$ is the golden ratio, thereby placing the work
of Knill and Tangerman~\cite{Knill2011} on a rigorous footing. We
now discuss directions for further research suggested by the work
presented here and by the extensive discussion of open questions in~\cite{Knill2011}.

In studies of quasi-periodic dynamics, it is usual to proceed from
the golden mean, through quadratic irrationals to more general irrationals,
sometimes with arithmetic conditions to overcome small-divisor obstructions.
Sudler's sine product is well suited for such an approach. It is likely
that the methods in sections~\ref{sec:Overview}--\ref{sec:EndMain}
may be extended to all quadratic irrationals (with appropriate modification
to take account of the (eventual) periodicity of the continued fraction
expansion) and may provide a foundation to study the case of arbitrary
irrational $\omega$, leading to refinements of the norm and peak
results presented in section \ref{sec:Survey}.

Indeed, following Knill and Tangerman~\cite{Knill2011}, we conjecture
that for quadratic $\omega$ with a period $\ell$ continued fraction
($\ell\ge1$) and with rational convergents $p_{k}/q_{k}$, the Sudler
product for $n=q_{k}$ will converge to a periodic sequence of period
dividing $\ell$. Moreover, for $\omega$ satisfying other suitable
arithmetic conditions such as Diophantine or Brjuno, the Sudler product
will be bounded for $n=q_{k}$, where $p_{k}/q_{k}$ are the rational
convergents of $\omega$. It is likely that a renormalisation approach
will elucidate the overall structure of the Sudler product for arbitrary
irrational $\omega$.

Sudler's sine product appears in several areas of pure and applied
mathematics. In the dynamical context, it arises in the renormalisation
analysis of strange non-chaotic attractors for zero phase (see~\cite{Kuznetsov1995a}).
An analysis of non-zero phase leads to the more complex product 
\begin{equation}
\prod_{r=1}^{n}2\sin\pi(r\omega+\alpha)\,,
\end{equation}
the analysis of which appears difficult in general, although we studied
a special case (in which $\alpha$ decreases with $n$) in section
\ref{sec:Consequences} of this paper. Both the special case and the
general case require obtaining the growth rate of the series 

\begin{equation}
\sum_{r=1}^{F_{n}}\cot\pi r\omega\,.
\end{equation}

We studied this growth rate in the case of the golden mean also in
section \ref{sec:Consequences}. Again the results need to be generalised
along the lines of the programme outlined above.

Coupled with the work of Knill and Lesieutre \cite{Knill2012} (in
which they used Herman's Denjoy-Koksma result \cite{HermanDenjoy1979}
to study the generalised product 
\begin{equation}
P_{n}(f,\omega)=\prod_{r=1}^{n}|f(r\omega)|\label{eq:gen_prod}
\end{equation}
where $\log|f|$ is of bounded variation and $\omega$ is Diophantine),
our results suggest that a fruitful research direction would be to
adapt the methods in this paper to study in detail the product~\eqref{eq:gen_prod},
first in the golden mean case and then for more general irrationals.
It is likely that the symmetry properties of $f$ will prove important
in the application of the methods presented here.

\part*{Appendices}

\appendix

\section{\label{sec:SineAppendix}Basic results on sums and products of sines}

\subsection{Sums of sines}

We shall need the following result of Lagrange:

For $e^{ix}\ne1$, ie $x\ne2r\pi$,
\[
\sum_{k=1}^{n}e^{i(\theta+kx)}=e^{i(\theta+x)}\frac{e^{inx}-1}{e^{ix}-1}=e^{i\theta}e^{ix(n+1)/2}\frac{\sin nx/2}{\sin x/2}
\]

Equating imaginary parts

\begin{eqnarray}
\sum_{k=1}^{n}\sin(\theta+kx) & = & \sin(\theta+\frac{(n+1)x}{2})\sin\frac{nx}{2}/\sin\frac{x}{2}\nonumber \\
 & = & \frac{\cos(\theta+\frac{x}{2})-\cos(\theta+(n+\frac{1}{2})x)}{2\sin\frac{x}{2}}\label{eq:SumSin}
\end{eqnarray}
And, applying the transform $\theta\longmapsto\theta+\pi/2$:
\begin{equation}
\sum_{k=1}^{n}\cos(\theta+kx)=\frac{\sin(\theta+(n+\frac{1}{2})x)-\sin(\theta+\frac{x}{2})}{2\sin\frac{x}{2}}
\end{equation}

We will also need the related result obtained by differentiating the
previous identity:

\begin{eqnarray*}
\sum_{k=1}^{n}-k\sin(\theta+kx) & = & \frac{(n+\frac{1}{2})\cos(\theta+(n+\frac{1}{2})x)-\frac{1}{2}\cos(\theta+\frac{x}{2})}{2\sin\frac{x}{2}}-\frac{\left(\sin(\theta+(n+\frac{1}{2})x)-\sin(\theta+\frac{x}{2})\right)\frac{1}{2}\cos\frac{x}{2}}{2\sin^{2}\frac{x}{2}}
\end{eqnarray*}

Simplifying:

\begin{equation}
\sum_{k=1}^{n}k\sin(\theta+kx)=\frac{\sin(\theta+nx)-\sin\theta-2n\cos(\theta+(n+\frac{1}{2})x)\sin\frac{x}{2}}{4\sin^{2}\frac{x}{2}}\label{eq:SumNSin}
\end{equation}

\subsection{Products of sines}

We now derive a number of preliminary results from the factorisation:

\begin{eqnarray}
z^{2n}-2\cos(\theta)z^{n}+1 & = & (z^{n}-e^{i\theta})(z^{n}-e^{-i\theta})\\
 & = & \prod_{r=0}^{n-1}(z-e^{i(\theta+2\pi r)/n})(z-e^{-i(\theta+2\pi r)/n})\nonumber \\
 & = & \prod_{r=0}^{n-1}(z^{2}-2\cos(\frac{\theta+2\pi r}{n})\, z+1)\nonumber 
\end{eqnarray}

Putting $\phi=2n\theta$, evaluating at $z=1$ gives $4\sin^{2}(n\phi)=\prod_{0}^{n-1}4\sin(\phi+r\pi/n)$,
which after consideration of signs gives 
\begin{equation}
\prod_{0}^{n-1}2\sin(\phi+\frac{\pi r}{n})=2\sin(n\phi)\label{eq:GenAnglen}
\end{equation}

Taking logs and differentiating with respect to $\phi$ gives
\begin{equation}
\sum_{r=0}^{n-1}\cot(\phi+\frac{\pi r}{n})=n\cot n\phi
\end{equation}

Similarly evaluating at $z=-1$ gives
\begin{eqnarray}
\prod_{r=0}^{n-1}2\cos(\phi+r\pi/n)=(-1)^{\frac{n-1}{2}}2\cos(n\phi) &  & n\,\textrm{odd}\label{eq:OddAnglen}\\
\prod_{r=0}^{n-1}2\cos(\phi+r\pi/n)=(-1)^{\frac{n}{2}}2\sin(n\phi) &  & n\,\textrm{even}\label{eq:EvenAnglen}
\end{eqnarray}

And combining \eqref{eq:GenAnglen} with \eqref{eq:OddAnglen},\eqref{eq:EvenAnglen}
gives (substituting the symbol $\phi$ for $2\phi$) 

\begin{eqnarray}
\prod_{r=0}^{n-1}2\sin(\phi+2r\pi/n)=(-1)^{\frac{n-1}{2}}2\sin(n\phi) &  & n\,\textrm{odd}\label{eq:OddDouble}\\
\prod_{r=0}^{n-1}2\sin(\phi+2r\pi/n)=(-1)^{\frac{n}{2}}2(1-\cos(n\phi)) &  & n\,\textrm{even}\label{eq:EvenDouble}
\end{eqnarray}

These identities vanish when $\phi=r\pi$, but dividing by the $r=0$
term and allowing $\phi\rightarrow0$, we obtain:

\begin{eqnarray}
\prod_{r=1}^{n-1}2\sin(\frac{\pi r}{n}) & = & n\label{eq:lemSinProduct}\\
\prod_{r=1}^{n-1}2\sin(\frac{2r\pi}{n}) & = & (-1)^{\frac{n-1}{2}}n\qquad n\,\textrm{odd}\\
\prod_{\begin{array}{c}
r=1\\
r\ne n/2
\end{array}}^{n-1}2\sin(\frac{2r\pi}{n}) & = & (-1)^{\frac{n}{2}-1}\left(\frac{n^{2}}{4}\right)\qquad n\,\textrm{even}
\end{eqnarray}

\section{Fibonacci Numbers}

\noindent We will make use of some standard results about Fibonacci
numbers $F_{n}$ (defined for $n\ge0$ by $F_{0}=0,F_{1}=1$ and $F_{n+1}=F_{n}+F_{n-1}$),
which we quote without proof:
\begin{eqnarray}
F_{n+1}F_{n-1}-F_{n}^{2} & = & (-1)^{n}\label{eq:FibProduct}\\
F_{n}\textrm{ is even iff }n & = & 3k\textrm{ for some }k\ge0\label{eq:FnEven}\\
F_{n} & = & \frac{1}{\sqrt{5}}\left(\omega^{-n}-(-\omega)^{n}\right)\label{eq:Fn=00003D}\\
F_{n}\omega & = & F_{n-1}-(-\omega)^{n}\label{eq:Fnw}\\
F_{n}\omega^{n} & = & \frac{1}{\sqrt{5}}\left(1-(-1)^{n}\omega^{2n}\right)\label{eq:Fnwn}\\
\omega & = & \frac{F_{n-1}}{F_{n}}-\frac{(-\omega)^{n}}{F_{n}}\label{eq:w}
\end{eqnarray}

We recast the previous two results for estimating purposes (noting
$0<\omega<1$) as
\begin{eqnarray}
F_{n} & = & \frac{\omega^{-n}}{\sqrt{5}}+O(F_{n}^{-1})\label{eq:FnEst}\\
\omega & = & \frac{F_{n-1}}{F_{n}}+O(\omega^{2n})\label{eq:wEstimate}
\end{eqnarray}

\bibliographystyle{unsrt}
\bibliography{Verschueren}

\begin{thebibliography}{10}

\bibitem{Sudler1964}
C.~Sudler, Jr.
\newblock An estimate for a restricted partition function.
\newblock {\em Quart. J. Math. Oxford Ser. (2)}, 15:1--10, 1964.

\bibitem{Wright1964}
E.~M. Wright.
\newblock Proof of a conjecture of {S}udler's.
\newblock {\em Quart. J. Math. Oxford Ser. (2)}, 15:11--15, 1964.

\bibitem{Lubinsky1998}
D.~S. Lubinsky.
\newblock The size of {$(q;q)_n$} for {$q$} on the unit circle.
\newblock {\em J. Number Theory}, 76(2):217--247, 1999.

\bibitem{Grebogi1984}
Celso Grebogi, Edward Ott, Steven Pelikan, and James~A. Yorke.
\newblock {Strange attractors that are not chaotic}.
\newblock {\em Physica D: Nonlinear Phenomena}, 13(1-2):261--268, August 1984.

\bibitem{Kuznetsov1995a}
Sergey Kuznetsov, Arkady Pikovsky, and Ulrike Feudel.
\newblock {Birth of a strange nonchaotic attractor: A renormalization group
  analysis}.
\newblock {\em Physical Review E}, 51(3):R1629--R1632, March 1995.

\bibitem{Knill2011}
Oliver Knill and Folkert Tangerman.
\newblock Self-similarity and growth in {B}irkhoff sums for the golden
  rotation.
\newblock {\em Nonlinearity}, 24(11):3115--3127, 2011.

\bibitem{Awata01052013}
Hidetoshi Awata, Shinji Hirano, and Masaki Shigemori.
\newblock {The partition function of ABJ theory}.
\newblock {\em Progress of Theoretical and Experimental Physics}, 2013(5),
  2013.

\bibitem{Knill2012}
Oliver Knill and John Lesieutre.
\newblock Analytic continuation of {D}irichlet series with almost periodic
  coefficients.
\newblock {\em Complex Anal. Oper. Theory}, 6(1):237--255, 2012.

\bibitem{Erdos1959}
P.~Erd{\H{o}}s and G.~Szekeres.
\newblock On the product {$\Pi^n_{k=1}(1-z^ak)$}.
\newblock {\em Acad. Serbe Sci. Publ. Inst. Math.}, 13:29--34, 1959.

\bibitem{Freiman1988}
G.~Freiman and H.~Halberstam.
\newblock On a product of sines.
\newblock {\em Acta Arith.}, 49(4):377--385, 1988.

\bibitem{Bell1998}
J.~P. Bell, P.~B. Borwein, and L.~B. Richmond.
\newblock Growth of the product {$\prod^n_{j=1}(1-x^{a_j})$}.
\newblock {\em Acta Arith.}, 86(2):155--170, 1998.

\bibitem{Bell2013}
Jordan Bell.
\newblock Estimates for the norms of products of sines and cosines.
\newblock {\em Journal of Mathematical Analysis and Applications}, 405(2):530
  -- 545, 2013.

\bibitem{LubinskyPade1987}
D.~S. Lubinsky and E.~B. Saff.
\newblock Convergence of {P}ad\'e approximants of partial theta functions and
  the {R}ogers-{S}zeg{\H o} polynomials.
\newblock {\em Constr. Approx.}, 3(4):331--361, 1987.

\bibitem{HermanDenjoy1979}
Michael-Robert Herman.
\newblock Sur la conjugaison diff\'erentiable des diff\'eomorphismes du cercle
  \`a des rotations.
\newblock {\em Inst. Hautes \'Etudes Sci. Publ. Math.}, (49):5--233, 1979.

\end{thebibliography}

\end{document}